\documentclass{amsart}
\usepackage{enumerate}
\usepackage{mathrsfs}
\usepackage{amsfonts,amssymb,amsmath}
\usepackage{amsthm}
\usepackage{epstopdf}	
\usepackage{float}
\usepackage{caption}
\usepackage{tikz}
\usepackage{subfigure}
\makeatletter
\@addtoreset{equation}{section}

\makeatother
\newtheorem{thm}{Theorem}[section]
\newtheorem{lem}{Lemma}[section]

\newtheorem{prop}{Proposition}[section]
\newtheorem{remark}{Remark}[section]

\newcommand{\R}{\mathbb{R}}

\begin{document}
\title[A degenerate diffusion equation with a multistable reaction]{Asymptotic Behavior of Solutions of a Degenerate Diffusion Equation with a Multistable Reaction$^\S$}
\thanks{$\S$ This research was supported by the National Natural Science Foundation of China (12071299, 12471199).}
\author[F. Li, B. Lou]
{Fang Li$^{\dag}$,\ Bendong Lou$^{\dag, *}$}
\thanks{$\dag$ Mathematics and Science College, Shanghai Normal University, Shanghai 200234, China.}
\thanks{{\bf Email:} {\sf lifwx@shnu.edu.cn} (F. Li), {\sf lou@shnu.edu.cn} (B. Lou).}
\thanks{$^*$ Corresponding author.}
\date{\today}

\begin{abstract}
We consider a generalized degenerate diffusion equation with a reaction term $u_t=[A(u)]_{xx}+f(u)$, where $A$ is a smooth function satisfying $A(0)=A'(0)=0$ and $A(u),\ A'(u),\ A''(u)>0$ for $u>0$, $f$ is of monostable type in $[0,s_1]$ and of bistable type in $[s_1,1]$. We first give a trichotomy result on the asymptotic behavior of the solutions starting at compactly supported initial data, which says that, as $t\to \infty$, either small-spreading (which means $u$ tends to $s_1$), or transition, or big-spreading (which means $u$ tends to $1$) happens for a solution. Then we construct the classical and sharp traveling waves (a sharp wave means a wave having a free boundary which satisfies the Darcy's law) for the generalized degenerate diffusion equation, and then using them to characterize the spreading solution near its front.
\end{abstract}


\maketitle

\baselineskip 17pt


\section{Introduction}
In this paper, we consider the following Cauchy problem
\begin{equation*}\label{p}
\begin{cases}
u_t=[A(u)]_{xx}+f(u),~  &x\in\mathbb{R},~t>0,\cr
u(0,x)=u_0(x),~        &x\in\mathbb{R},
\end{cases}
\tag{P}
\end{equation*}
where $f$ and $u_0$ will be specified below, $A$ belongs to the following class
\begin{equation*}\label{ass-A}
\mathcal{A}:= \left\{ A \left|
 \begin{array}{l}
 A\in C^1([0,\infty)) \cap C^\infty ((0,\infty)), \  A(u), A'(u), A''(u)>0 \mbox{ for }u>0,\\
A(0)=A'(0)=0,\  \int_0^1 \frac{A'(r)}{r} dr<\infty,\ \frac{rA''(r)}{A'(r)} \to A_* >0\ \mbox{ as } r\to 0+0
 \end{array}
 \right.
 \right\}.
\tag{A}
\end{equation*}
A typical example of $A(u)$ is $u^m\ (m>1)$, which corresponds the porous medium equation and
is used in particular to model the population dynamics where the diffusion depends on the population density  (cf. \cite{A2, Vaz-book, WuYin-book}). Other examples include $A(u)=u^{3/2}+u^2$, $A(u)=u^2 \ln (u+1)$, etc., the equation is degenerate at $u=0$ in any of these examples.

When $A(u)\equiv u$, the diffusion term in \eqref{p} is a non-degenerate linear one. The corresponding reaction diffusion equations have been extensively studied. In particular, the equations with typical monostable, bistable and combustion types of reactions have been understood very well (cf. \cite{AW1, AW2, DM, FM, F, KPP, Z} etc.).
In recent years, the linear diffusion equations with multistable reactions have also been studied. One example is the following monostable-bistable type of reaction:
\begin{equation*}\label{ass-f}
\left\{
\begin{array}{l}
f(u)\in C^2([0,\infty)),\ \mbox{it has exactly four zeros}: 0<s_1 <s_2<1,\
f'(0)>0>f'(s_1), f'(1)<0, \\
f \mbox{ is of monostable type in } [0,s_1] \mbox{ and of bistable type in } [s_1,1], \ \int_{s_1}^1  A'(s) f(s) ds >0.
\end{array}
\right.
\tag{F}
\end{equation*}
(The $C^2$ assumption is only used to derive the lower bound of $u_{xx}$ at the free boundaries when we consider the well-posedness, see \cite{LZ, Vaz-book}. In other places we only require that $f$ is $C^1$ or Lipschitz.) Reaction diffusion equations with such reactions were considered by Ludwig, Aronson, and Weinberger \cite{LAW} in a population model on the spruce budworm. Kawai and Yamada \cite{KY}, Kaneko, Matsuzawa and Yamada \cite{KMY-SIAM, KMY-JMPA} also studied such equations with free boundaries. In particular, they specified the asymptotic behavior of solutions.
Recently, Polacik \cite{Polacik-MAMS} gave a very complete analysis on linear diffusion equations with generic multistable reactions. When $f$ is a multistable heterogeneous reaction, the long time behavior of the solutions have also been considered by Ducrot, Giletti and Matano \cite{DGM} (for spatially periodic case) and by Ding and Matano \cite{DingMatano} (for temporally periodic case). In particular, \cite{DGM} proposed the concept {\it propagating terrace} to characterize the spreading phenomena of the solutions in such equations.

In some cases, the diffusion rate depends on the density itself. So the corresponding equation may involve a nonlinear-diffusion term. A typical example of such kind is the following porous medium equation with a reaction (RPME, for short):
\begin{equation}\label{PME}
u_t = (u^m)_{xx} + f(u),\qquad x\in \R,\ t>0,
\end{equation}
where $m>1$, implying that the diffusion is a nonlinear, degenerate and slow one. When $f$ is a monostable reaction, the equation \eqref{PME} was considered as a population model by Aronson \cite{A2}. Later, this kind of equations were studied as models in population dynamics, combustion problems, propagation of interfaces and nerve impulse propagation phenomena in porous media, as well as the propagation of intergalactic civilizations in the field of astronomy (cf. \cite{GN, GM, NS, SGM} etc.). Recently, Lou and Zhou \cite{LZ} gave a general convergence result for bounded solutions of \eqref{PME} by using a generalized zero number argument. Their results imply that any bounded solution converges as $t\to \infty$ to a stationary one, and so they could use this general result to specify the asymptotic behavior for solutions of \eqref{PME} with monostable, bistable and combustion reactions. In the last years, Du,  Quir\'{o}s and Zhou \cite{DQZ}, and G\'{a}rriz \cite{G} used the sharp wave to characterize the spreading solutions of \eqref{PME}.

In this paper, we will study the problem \eqref{p} with a monostable-bistable reaction, that is, a generic degenerate diffusion equation with multistable reaction.  We will specify the asymptotic behavior of the solutions of the Cauchy problem. We choose initial data from the following set: for some $b>0$,
\begin{equation*}\label{init-val}
u_0\in \mathcal{I}(b):=\{\phi~|~\phi\in C(\mathbb{R}),~\phi>0 \text{ in } (-b,b), \text{ and }	\phi\equiv0 \text{ otherwise}\}.
\tag{I}
\end{equation*}
Due to the degeneracy of the equation, the Cauchy problem \eqref{p} with $u_0\in \mathcal{I}(b)$ has only a weak solution, which is defined as the following (cf. \cite{Knerr, Vaz-book, WuYin-book}): for any $T>0$, denote $Q_T:=\mathbb{R}\times(0,T)$, then a function $u(x,t)\in C(Q_T)\cap L^\infty (Q_T)$ is called a {\it very weak solution} of \eqref{p} if for any $\varphi \in C^\infty_c (Q_T)$, there holds
\begin{equation}\label{def-weak sol}
\int_{\mathbb{R}} u(x,T)\varphi(x,T) dx  =  \int_{\mathbb{R}} u_0(x) \varphi(x,0)dx +
\iint_{Q_T} f(u) \varphi  dx dt  + \iint_{Q_T} [u\varphi_t + A(u) \varphi_{xx}] dx dt.
\end{equation}
As an extension of the definition, if $u$ satisfies \eqref{def-weak sol} with inequality \lq\lq $\geq$" (resp. \lq\lq $\leq$"), instead of equality, for every test function $\varphi\geq 0$, then $u$ is called a {\it very weak supersolution} (resp. {\it very weak subsolution}) of \eqref{p} (cf. \cite[Chapter 5]{Vaz-book}). From \cite{ACP, Knerr, PV, S, Vaz-book} etc. we know that under the assumption \eqref{ass-A} and \eqref{ass-f}, the problem \eqref{p} with $u_0\in \mathcal{I}(b)$ has a unique very weak solution $u(x,t)\in C(Q_T)\cap L^\infty (Q_T)$ for any $T>0$, $u(x,t)\geq 0$ in $Q_T$,
and the support of $u(\cdot,t)$ is an interval between the left free boundary $l(t)$ and the right one $r(t)$. In addition, after some waiting times $t_*(\pm b)$, both free boundaries move outward monotonically according to the Darcy's law:
$$
r'(t)= - v_x(r(t)-0,t) \mbox{ for }t>t_*(b),\qquad l'(t) =-v_x(l(t)+0,t) \mbox{ for } t>t_*(-b),
$$
where $v(x,t)$ denotes the (generalized) pressure function corresponding to $u(x,t)$:
\begin{equation}\label{def-pressure}
v(x,t) = \Lambda(u(x,t)) := \int_0^{u(x,t)} \frac{A'(r)}{r} dr.
\end{equation}
(In this paper we always use $v_x(a-0,t)$/$v_x(a+0,t)$ to denote the left/right derivative of $v$ at $x=a$.) The Darcy's law is a well known property for the PME or nonlinear, degenerate diffusion equations without reactions (cf. \cite{A2, Knerr, Vaz-book}). For the equations with reactions, it is
also believed to be true. In fact, it was derived in detail in \cite{LZ} for the PME with a reaction, the derivation remains valid for other nonlinear, degenerate diffusion equations as in \eqref{p}.

We first show a trichotomy result on the asymptotic behavior of solutions of \eqref{p}:

\begin{thm}\label{thm1}
Assume \eqref{ass-A} and \eqref{ass-f}. Let $\phi\in \mathcal{I}(b)$ and $u_\sigma(x,t)$ be the very weak solution of \eqref{p} with $u_0=\sigma\phi$ for $\sigma>0$. Then there exists $\sigma_* \in (0,\infty]$ such that the following trichotomy result holds:
\begin{enumerate}[{\rm (i)}]
\item {\rm Small spreading:} for $0<\sigma<\sigma_*$, $u_\sigma(\cdot,t)\to s_1$ as $t\to \infty$, in the topology of $L^\infty_{loc}(\mathbb{R})$.
\item {\rm Transition:} when $\sigma_* <\infty$, $u_{\sigma_*}(\cdot,t)\to U^* (\cdot -x_0)$ as $t\to \infty$, in the topology of $L^\infty_{loc}(\mathbb{R})$, where $x_0\in [-b,b]$ and $U^*\in C^2(\R)$ is an even stationary solution with $U^*(\pm \infty)=s_1$ and $(U^*)'(x)<0$ for $x>0$.

\item {\rm Big spreading:} when $\sigma_*<\infty$, $u_\sigma(\cdot,t)\to 1$ as $t\to \infty$ for any $\sigma>\sigma_*$, in the topology of $L^\infty_{loc}(\mathbb{R})$.
\end{enumerate}
In addition, in any of these cases, $r(t),-l(t)\to \infty$ as $t\to \infty$.
\end{thm}

This theorem specifies the asymptotic behavior of the solutions in $L^\infty_{loc} (\mathbb{R})$ topology. Another interesting problem is to study the asymptotic profile of a solution near its free boundaries and/or near its interior fronts. As in the linear diffusion equations, the traveling waves and propagating terrace play important roles in such problems. For our equation we call a solution $u$ as a traveling wave if it has the form $u(x,t)=Q_c(x-ct)=Q_c(\xi)$ for some speed $c>0$ and some profile $Q_c$, where $Q_c \in C(\mathbb{R})$, and
\begin{align}\label{trav-equ}
    \int_{\mathbb{R}} [A(Q_c) \varphi''- cQ_c\varphi'+f(Q_c)\varphi] d\xi=0,\qquad \varphi \in C_c^\infty(\mathbb{R}).
\end{align}
By using a phase plane analysis (see details in the last section, see also \cite{A2} for RPMEs), we will see that our equation has two kinds of (rightward propagating) traveling waves, all of them have decreasing profiles. Each of the first kind has a positive profile and connects two steady states as the following
\begin{align}\label{trav-cond}
Q_c(-\infty)=s_1\ \ (\mbox{resp.,}\ 1,\ 1), \quad Q_c(+\infty)=0\ \ (\mbox{resp.,}\ 0,\ s_1),\quad \mbox{ and } Q'_c(\xi)<0 \mbox{ for } \xi \in \R.
\end{align}
In particular, the traveling wave connecting $1$ at $\xi=-\infty$ and $s_1$ at $\xi=+\infty$, denoted by $Q_{c_z} (x-c_z t)$, will be used to describe the {\it propagating terrace} phenomenon in the range $(s_1, 1)$. Each of the second kind, however, has only positive profile on the half line:
\begin{align}\label{sharp-wave}
\left\{
\begin{array}{l}
Q_c(-\infty)=s_1\ \ (\mbox{resp.,}\ 1), \quad Q_c(\xi)=0 \mbox{ for } \xi \geq  0, \quad
Q'_c(\xi)<0 \mbox{ for } \xi<0,\\
\mbox{and it satisfies the Darcy's law}: \quad c= - [\Lambda(Q_c)]'(0-0).
\end{array}
\right.
\end{align}
This kind of traveling waves are called {\it sharp waves}. When $Q_c(-\infty)=s_1$, the sharp wave, denoted by $Q_{c_s}(x-c_s t)$, is unique, and $c_s>0$. It is called a {\it small sharp wave} for convenience. Similarly, when $Q_c(-\infty)=1$, the sharp wave is also unique (if it exists, see details in the last section), which is denoted by $Q_{c_b}(x-c_b t)$ and is called a {\it big sharp wave}.

The following result shows that the small sharp wave can be used to describe the asymptotic behavior of the solution in the range $u\in (0,s_1)$ when the small spreading, transition or some cases of the big spreading happens.

\begin{thm}\label{thm2}
Assume \eqref{ass-A} and \eqref{ass-f}. Let $u(x,t)$ be a global solution of \eqref{p} with $u_0\in \mathcal{I}(b)$. Then there exist constants $l^*, r^*\in\mathbb{R}$ such that
\begin{equation}\label{s-profile}
 \lim_{t\to\infty}\sup_{x\in [H(t),+\infty)}|u(x+r(t),t)-Q_{c_s}(x)|=0, \ \
    \lim_{t\to\infty} [r(t)-c_s t]=r^*,
\end{equation}
\begin{equation}\label{s-profile-left}
 \lim_{t\to\infty}\sup_{x\in(-\infty, -H(t)]}|u(x+l(t),t)-Q_{c_s}(- x)|=0, \ \
    \lim_{t\to\infty}[l(t)+c_s t]=l^*,
\end{equation}
if one of the following cases holds:
\begin{enumerate}[{\rm (i)}]
\item the small spreading  happens, and $H(t)\equiv -c_s t$;
\item $u$ is a transition solution, and $H(t)=(\delta -c_s) t$ for any small $\delta >0$;
\item the big spreading  happens, $c_s >c_z$ and $H(t)=(c-c_s)t$ for any $c\in (c_z, c_s)$.
\end{enumerate}

In addition, if the big spreading happens and $c_s =c_z$, then $r(t) = c_s t + o(t),
\ l(t)=-c_s t +o(t)$, and, for any given number $H_0>0$, the first limit in
\eqref{s-profile} and that in \eqref{s-profile-left} hold for $H(t)\equiv -H_0$.

\end{thm}

When the big spreading  happens, we have further description of $u$ in the range where $u\in (s_1, 1)$, which are stated as the following.

\begin{thm}\label{thm3}
Assume \eqref{ass-A} and \eqref{ass-f}. Let $u(x,t)$ be a big spreading solution.
\begin{enumerate}[{\rm (i)}]
\item When $c_s<c_z$, there exist constants $l^*_1, r^*_1\in\mathbb{R}$ such that
\begin{align}\label{b1-profile-lr}
	&\lim_{t\to\infty}\sup_{x\in[-c_b t,+\infty)}|u(x+c_b t,t)-Q_{c_b}(x -r^*_1)|=0, \ \
	\lim_{t\to\infty}[r(t)-c_bt]=r^*_1, \\	
	&\lim_{t\to\infty}\sup_{x\in(-\infty,c_b t]}|u(x-c_b t,t)-Q_{c_b}(l^*_1-x)|=0, \ \
	\lim_{t\to\infty}[l(t)+c_bt]=l^*_1.
\end{align}	

\item When $c_s >c_z$, there exist $l^*_2, r^*_2\in \R$ such that, for any $c\in(c_z,c_s)$, as $t\to \infty$,  there holds
\begin{align}
\sup_{x\in[-c_z t ,(c-c_z)t ]}|u(x+c_z t,t)-Q_{c_z}(x- r^*_2)| \to 0, \qquad 	
\sup_{x\in[(c_z -c)t,c_z t]}|u(x-c_z t,t)-Q_{c_z}(l^*_2 - x)| \to 0.\label{b2-profile-l}
\end{align}

\item
When $c_s =c_z$, there exist $\theta_\pm (t)$ with $\theta'_\pm (t)\to 0\ (t\to \infty)$ such that
\begin{equation}\label{cs=cz=upper}
u(x+c_z t +\theta_+(t), t)\to Q_{c_z}(x),\qquad
u(x-c_z t -\theta_-(t), t)\to Q_{c_z}(-x)\quad  \mbox{\ \ as\ \ }t\to \infty,
\end{equation}
in the topology of $C^{2}_{loc}(\R)$.

\end{enumerate}
\end{thm}

The rest of this paper is organized as follows. In Section 2 we state some preliminary results including the properties of weak solutions, classification of nonnegative stationary solutions, traveling waves. The proof of trichotomy results in Theorem \ref{thm1} is shown in Section 3. In Section 4, we will prove Theorem \ref{thm2}, that is, to use the small sharp wave $Q_{c_s}(x-c_s t)$ to characterize the solutions in the range $u\in (0,s_1)$. Section 5 is devoted to prove Theorem \ref{thm3}, that is, to use the traveling waves $Q_{c_z}(x-c_z t)$ and $Q_{c_b}(x-c_b t)$ to characterize the solutions in the range $u\in (s_1,1)$.
In the last section, we provide the detailed construction of stationary solutions and traveling waves by using the phase plane analysis.


\section{Preliminaries}
In this section, we give some preliminary results, including some properties of weak solutions, nonnegative stationary solutions and traveling waves. We also give an analogue convergence result for bounded solutions as in \cite{LZ} for RPMEs.

\subsection{Basic results}
For a very weak solution $u(x,t)$ of \eqref{p} with initial data taking from $\mathcal{I}(b)$, we have the following facts.
\begin{enumerate}[(a).]
\item {\it Generalized pressure function}. As in PMEs, we use not only the function $u(x,t)$ but also the function $v(x,t)$ defined by \eqref{def-pressure}. In the PMEs, $u$ corresponds to the density and $v$ corresponds to the pressure of a gas, respectively, so we also call $v$ a generalized pressure function. Denote
\begin{equation}\label{u-to-v}
\lambda(v) := \Lambda^{-1} (v) \Leftrightarrow \Lambda(\lambda(v))\equiv v,\quad B(v) := A'(\lambda(v)),\quad h(v):= \frac{f(\lambda(v))}{\lambda'(v)},\quad v>0.
\end{equation}
We make the supplementary definitions: $\Lambda (0)=\lambda (0)=0$ and $h(0)=0$. Then
\begin{equation}\label{prop-B-h}
\left\{
\begin{array}{l}
B(0)=0,\quad B'(0+0)=A_*,\quad B= \frac{\lambda(v)}{\lambda'(v)} \in C^1([0,\infty)),\quad B'(v)>0\mbox{ for }v>0,\\
h'(0+0)= f'(0)A_*,\quad h\in C^1([0,\infty))
\end{array}
\right.
\end{equation}
where $A_*$ is the limit in \eqref{ass-A} (see the detailed proof in the last section). By using $v$ we can rewrite the problem \eqref{p} as
\begin{equation}\label{p-v}
\left\{
\begin{array}{ll}
v_t = B(v) v_{xx} + v^2_x + h(v),&  x\in \R,\ t>0,\\
v(x,0)= v_0(x) := \Lambda(u_0(x)), & x\in \R.
\end{array}
\right.
\end{equation}

\item {\it Free boundaries}. A right free boundary $r(t)$ and a left one $l(t)$ appear in the solution,  $u(x,t)$ is positive in $(l(t),r(t))$ and is zero outside this interval.
\item {\it Waiting time and Darcy's law}. The right/left free boundary has a waiting time $t_*(b)/t_*(-b)\in [0,\infty]$, and $t_*(b)=0$ if $v'_0(b-0)<0$. After the waiting time, the free boundaries satisfy the Darcy's law:
\begin{equation}\label{general-sol-Darcy}
r'(t) = - v_x  (r(t)-0,t) \mbox{\ for\ } t>t_*(b),\qquad l'(t) = - v_x (l(t)+0,t)\mbox{ for } t>t_*(-b).
\end{equation}
By \cite[Proposition A.8]{LZ}, for any $t\geq t_*(b)+1$, $v(x,t)$ is Lipschitz in $x$ 
with Lipschitz number $L$ independent of $t$. So, $|v(r(t),t)-v(r(t)-y,t)|\leq L y$ for $y\in [0,1]$. This implies that 
\begin{equation}\label{r'l'-bound}
|r'(t)|=|v_x(r(t)-0,t)|\leq L \mbox{ for } t\geq t_*(b)+1.\quad \mbox{(Similarly,\ \  }
|l'(t)|\leq L' \mbox{ for } t\geq t_*(-b)+1.)
\end{equation}

Sometimes we need to consider a solution in a moving frame. For example, let $u(x,t)$ be a solution of \eqref{p} with free boundaries $l(t)<r(t)$, we may consider the function $\tilde{u}(x,t):= u(x-\beta t,t)$
with free boundaries $\tilde{l}(t):= l(t)+\beta t$, $\tilde{r}(t):= r(t) +\beta t$. Then $\tilde{u}$ satisfies
\begin{equation}\label{tilde-u-eq}
\tilde{u}_t = [A(\tilde{u})]_{xx} -\beta \tilde{u}_x + f(\tilde{u}),\qquad x\in \R,\ t>0,
\end{equation}
whose free boundaries satisfy the following modified Darcy's law:
\begin{equation}\label{tilde-Darcy}
\left\{
\begin{array}{l}
\tilde{r}'(t) = - [\Lambda(\tilde{u})]_x  (\tilde{r}(t)-0,0) +\beta \mbox{\ \ for\ \ } t>t_*(b),\\ \tilde{l}'(t) = - [\Lambda(\tilde{u})]_x (\tilde{l}(t)+0,0)+\beta \mbox{\ \  for\ \  } t>t_*(-b).
\end{array}
\right.
\end{equation}

\item {\it Positivity persistence and regularity}. For any $x_1\in \R$, once $u(x_1,t_1)>0$ for some $t_1\geq 0$, then $u(x_1,t)>0$ for all $t\geq t_1$, $u(x,t)$ is classical in a neighborhood of $(x_1,t_1)$ if $u(x_1,t_1)>0$.

In addition, we have the following locally uniform H\"{o}lder estimates (cf. \cite{DB-F0, DB-F} and \cite[Proposition A.8]{LZ}): for any $\tau >0$ and any compact domain $D\subset \R\times (\tau, \infty)$, there are positive constants $C$ and $\alpha$, both depending only on $A$ and $\tau$, such that
\begin{equation}\label{v-Holder-est}
\|v\|_{C^{\alpha,\frac{\alpha}{2}}(D)}\leq C.
\end{equation}

\item {\it Monotonicity outside the initial domain}. The solution is monotone in the following sense
$$
u_x(x,t)<0 \mbox{ for } b<x<r(t),\ t>t_*(b),\qquad
u_x(x,t)>0 \mbox{ for } l(t)<x<-b,\ t>t_*(-b).
$$
\end{enumerate}
The properties (a)-(d) can be found or be derived in a similar way as in the standard theory of PMEs, see, for example \cite{Knerr, Vaz-book,  WuYin-book}) and references therein. The last monotonicity property was proved in \cite{LZ} for RPMEs, which clearly remains true for the solutions of the problem \eqref{p}.

\subsection{Nonnegative stationary solutions}
We list all nonnegative stationary solutions of the equation in \eqref{p} (the detailed construction for these solutions is given in the last section by using a phase plane).
Such a solution means a function $U\in C(\R)$ with $U(x)\geq 0$ and it satisfies $[A(U)]''+f(U)=0$ in the weak sense. In addition, $U(x)>0$ in some interval $I$ where $U$ is actually a $C^2$ function. Some of them will be used as the $\omega$-limit of the solution of \eqref{p}, or used for comparison.
Using a phase plane analysis we have the following situations (see details in the last section, see also \cite{A2, AW1}):
\begin{enumerate}
\item[{\bf Case~A.}]~{\it Constant solutions}: $U_{q}(x)\equiv q$ in $\R$, where $q\in \{0, s_{1}, s_{2}, 1\}$.

\item[{\bf Case B.}]~{\it Ground state solution $U^*\in C^2(\R)$}: $U^* (0)=\theta$ for $\theta\in (s_2, 1)$ satisfying $\int_{s_1}^\theta A'(r)f(r)dr=0$, $U^*(\pm \infty)=s_1$, $U^*(x)=U^* (-x)$ and $(U^*)'(x)<0$ for all $x>0$.

\item[{\bf Case C.}] {\it Solutions with compact supports}. There are two kinds of such solutions.
First, for any $q_{1}\in (0,s_{1})$, there is a unique even stationary solution $U_{q_{1}} \in C^2((-L_1,L_1))$ and $A(U_{q_1}) \in C^1([-L_1, L_1])$ for some $L_1>0$, and
$$
U_{q_1}(0)=q_1,\quad U_{q_1}(\pm L_1)= U'_{q_1}(0)=0, \quad U'_{q_1}(x)<0 \mbox{ for }0<x< L_1, \quad [A(U_{q_1})]'(L_1-0) \in (-\infty,0).
$$
Under the assumption $f'(0)>0$ we even have $L_1\to 0$ as $q_1\to 0$.
Next, for any $q_{2}\in (\theta,1)$ there is a unique even stationary solution $U_{q_{2}}(x)$ having similar properties as $U_{q_1}$, with $L_1$ replaced by some $L_2>0$.
We call $U_{q_1}$/$U_{q_2}$ a short/high stationary solution, respectively.

\item[{\bf Case D.}]~{\it Monotonic solutions}: $U_{1}^{\pm}(\cdot)$ with $U_1^+(\cdot) = U^-_1(-\cdot)\in C^2((0,\infty))\cap C(\R)$, and
$$
U_{1}^{+}(0)=0,\quad [A(U^+_1)]'(0+0)\in (0,\infty),\quad (U_{1}^{+})'(x)>0~\text{for}~x>0\mbox{ and } U_{1}^{+}(+\infty)=s_1;
$$
$U_2^{\pm}(x)$ with $U_2^+(x)=U^-_2(-x)$ has similar properties as $U^+_1$, but $U_2^+(+\infty)=1$.

\item[{\bf Case E.}]~{\it Periodic solutions}: for any $q\in (s_1, s_2)$, there exists a unique periodic solution $U_{per}(x)\in C^2(\R)$ with minimum $q$ and maximum in $(s_2,\theta)$.
\end{enumerate}

\begin{remark}\label{ss-eq-p}\rm
We must pay special attention to distinguish two concepts: {\it a stationary solution of the equation in its support $I\subset \R$} and {\it a stationary solution of the Cauchy problem \eqref{p}}. The former means a function $U(x)>0$ in some interval and satisfies $[A(U)]''+f(U)=0$ in a classical sense. The latter means a function $U(x)\geq 0$ such that the solution of \eqref{p} with initial data $U$ remains to be $U$ itself for all the time. This is equivalent to say that the waiting times of $U(x)$ on its free boundaries (if it has) are positive. Hence, the solutions in Cases C and D are not stationary ones of the Cauchy problem since their waiting times on the free boundaries are zero.
\end{remark}

\subsection{Traveling waves}
Using the phase plane analysis as in \cite{A2, AW1}, we can give the traveling waves and sharp waves of the equation in \eqref{p} as the following. The detailed construction is given in the last section by using a phase plane analysis, which is much more complicated than the phase plane analysis for reaction diffusion equations and porous media equations due to the general degenerate term $A(u)$.

\begin{prop}\label{prop:TWs}
Assume \eqref{ass-A} and \eqref{ass-f}. Then there exist positive constants $c_s>0$ and $c_z>0$ such that the equation \eqref{trav-equ} has traveling waves as the following.

\begin{enumerate}[{\rm (i)}]
\item It has a sharp wave $Q_{c_s}(x-c_s t)$, with $Q_{c_s}(-\infty)=s_1,\ Q'_{c_s}(\xi)<0$ for $\xi<0$, $Q_{c_s}(\xi)=0$ for $\xi\geq 0$, and $c_s=-[\Lambda(Q_{c_s})]'(0-0)$.
\item It has a unique traveling wave $Q_{c_z}(x-c_z t)$, with $Q_{c_z}(-\infty)=1,\ Q_{c_z}(+\infty)=s_1,\ Q'_{c_z}(\xi)<0$ for $\xi\in \R$;
\item In the case where $c_s<c_z$, there exists $c_b \in (c_s, c_z)$ such that the equation has a sharp wave $Q_{c_b}(x-c_b t)$, with $Q_{c_b}(-\infty)=1,\ Q'_{c_b}(\xi)<0$ for $\xi<0$, $Q_{c_b}(\xi)=0$ for $\xi\geq 0$, and $c_b=-[\Lambda(Q_{c_b})]'(0-0)$.
\end{enumerate}	
\end{prop}

\subsection{General convergence result}
Our problem \eqref{p} has only very weak solutions and they are classical only when they are positive. So it is convenient to define $\omega$-limit $w(x)$ of a solution $u(x,t)$ as the following:
\begin{equation}\label{def-conv}
\lim\limits_{n\to \infty} u(x,t_n)= w(x) \mbox{ means }
\left\{
 \begin{array}{l}
 u(x,t_n)\to w(x) \mbox{ in } C_{loc}(\R) \mbox{ topology};\\
 u(x,t_n)\to w(x) \mbox{ in } C^{2}_{loc}(J) \mbox{ topology if } w(x)>0 \mbox{ in }J.
 \end{array}
 \right.
\end{equation}
Recently, Lou and Zhou \cite[Theorem 1.1]{LZ} proved a general convergence result for RPMEs, which says that any bounded solution of a RPME starting at a compactly supported initial data converges in the sense of \eqref{def-conv} to a stationary solution of the Cauchy problem. Their proof depends essentially on the intersection number diminishing properties in the PMEs, as in the general convergence result for reaction diffusion equation in \cite{DM}. For our equation \eqref{tilde-u-eq} with initial data in $\mathcal{I}(b)$, we have the following analogue general convergence result.

\begin{thm}\label{thm:general-conv}
Assume \eqref{ass-A}, $f\in C^2([0,\infty)),\ f(0)=0$ and $\beta\in \R$. Let $\tilde{u}(x,t)$ be a bounded, nonnegative, time-global solution of \eqref{tilde-u-eq} with initial data in $\mathcal{I}(b)$. Then $\tilde{u}(\cdot,t)$ converges as $t\to \infty$, in the sense of \eqref{def-conv}, to a nonnegative stationary solution $\tilde{w}(x)$ of Cauchy problem of \eqref{tilde-u-eq}. More precisely,
\begin{enumerate}[{\rm (i)}]
\item when $\beta\not= 0$, $\tilde{w}$ is either $0$, or a positive function on $\R$, or a stationary solution of \eqref{tilde-u-eq} which is positive on the half line and has a boundary. In addition, in the last case, it has only one left boundary $l$ when $\beta>0$ where $[\Lambda(\tilde{w})]_x(l+0)=\beta$,
    it has only one right boundary $r$ when $\beta<0$ where $[\Lambda(\tilde{w})]_x(r-0)=\beta$;

\item when $\beta=0$, $\tilde{w}$ is either a nonnegative zero of $f$, or a ground state solution $U^*(x-x_0)$ for some $x_0\in [-b,b]$ with $U^*(\pm \infty)>0$ as in Case B in subsection 2.2, or in the case where $f(u)<0$ for $0<u\ll 1$ and
\begin{equation}\label{type-1}
\lim\limits_{u\to 0} \frac{-f(u)}{A''(u) u^2} =0,
\end{equation}
$\tilde{w}$ might be a ground state solution $U^{**}(x-x_1)$ as in Case B, but with $U^{**}(\pm \infty)=0$.
\end{enumerate}
\end{thm}

\begin{proof}
Note that the intersection number diminishing properties used in \cite{LZ} for RPMEs remain
hold for the equation \eqref{tilde-u-eq}, where the Darcy's law is replaced by the modified one \eqref{tilde-Darcy}, which does not affect the intersection number properties. So one can use a similar argument as showing \cite[Theorem 1.1]{LZ} to prove the theorem. We omit the details here.
\end{proof}

\begin{remark}\label{rem:ground-1-2}\rm
The solutions $U^*$ and $U^{**}$ in (ii) are similar as that in Case B in subsection 2.2. They are called Type I ground state solutions in the RPMEs (cf. \cite{LZ}). The additional assumption \eqref{type-1} (which is equivalent to $\alpha>m$ in the case where $A(u)=u^m$ for $m>1$ and $f(u)\sim -u^\alpha$ as $u\to 0$) is actually used to exclude the possibility of the following Type II ground state solutions. More precisely, if $f(u)<0$ for $0<u\ll 1$ and, instead of \eqref{type-1}, $f$ satisfies
\begin{equation}\label{type-2}
\liminf\limits_{u\to 0} \frac{-f(u)}{A''(u) u^{2+\delta}} >0
\end{equation}
for some $\delta>0$, (In the case where $A(u)=u^m$ and $f(u)\sim - u^\alpha$, this condition is true when $\alpha <m$.) then we may have the so-called Type II ground state solutions (cf. \cite{LZ}):  $U_* (x-z_1) + \cdots +U_*(x-z_k)$ for some positive integer $k$ and some $z_i\in \R \ (i=1,\cdots,k)$ with $z_i+2 L \leq z_{i+1}\ (i=1,2,\cdots, k-1)$, where $U_*$ is a positive, symmetrically decreasing stationary solution in the interval $(-L,L)$ for some $L>0$, and $U_*(x)=0$ for $|x|\geq L$, $[\Lambda(U_*)]'(\pm L)=0$.
\end{remark}


\section{Trichotomy Result}
In this section we prove the trichotomy result in Theorem \ref{thm1}.

\medskip
\noindent
{\it Proof of Theorem} \ref{thm1}. By the general convergence result Theorem \ref{thm:general-conv}, we know that, any solution $u_\sigma (x,t)$ of \eqref{p} with initial data $u_0=\sigma \phi \in \mathcal{I}(b)$ converges as $t\to \infty$ to a nonnegative stationary solution of the Cauchy problem \eqref{p}, in the sense of \eqref{def-conv}. Since, as we have mentioned in Remark \ref{ss-eq-p}, the stationary solutions in Cases C and D are not stationary ones of the Cauchy problem, to show the trichotomy result in Theorem \ref{thm1} we only need to show that constant solutions $0,s_2$ in Case A, periodic solutions in Case E are not in the $\omega$-limit set $\omega(u_\sigma)$ of the solution $u_\sigma$.

\medskip
\noindent
{\it Step 1. To exclude $0$ from $\omega(u_\sigma)$}. This indeed follows from the hair-trigger effect (cf. \cite{AW1,G,LZ}). More precisely, when we take $q_1\in (0,s_1)$ sufficiently small, then $\sigma \phi (x)>U_{q_1}(x+x_0)$ for some $x_0$, where $U_{q_1}$ is a stationary solution in Case C. Since $U_{q_1}(x+x_0)$ is a time-independent (very weak) subsolution of \eqref{p}, by comparison we have $u_\sigma (x,t)\geq U_{q_1}(x+x_0)$. So, $0\not\in \omega(u_\sigma)$.

\medskip
\noindent
{\it Step 2. To exclude $s_2$ from $\omega(u_\sigma)$}. Let $U_{per}(x)$ be a stationary solution in Case E. By the intersection number properties developed in \cite{LZ}, we see that the zero number of $u_\sigma (x,t) -U_{per}(x)$ is finite for any $t>0$. So its limit can not be $s_2-U_{per}(x)$, which has infinitely many zeros.

\medskip
\noindent
{\it Step 3. To exclude the periodic solutions $U_{per}$ from $\omega(u_\sigma)$}. The proof is similar as  that in the previous step. We only need to switch the roles of $s_2$ and $U_{per}(x)$.

\medskip
Next we prove the sharp threshold result. Set
$$
\Sigma_{w}:=\{\sigma>0~|~u_{\sigma}(x,t)\to w~\text{\rm as}~t\to\infty~\text{\rm in the topology of }~L^\infty_{loc}(\R)\},
$$
where $w$ is $s_1$ or $1$, and denote $\Sigma_* := (0,\infty)\backslash (\Sigma_{s_1}\cap \Sigma_1)$.

\medskip
\noindent
{\bf Claim 1}. The set $\Sigma_{s_1} = (0,\sigma_*)$ for some $\sigma_* \in (0,\infty]$.

\noindent
For any $\sigma_1$ satisfying $u_{\sigma_1}(x,t_1)< s_2$ for some $t_1\geq 0$, by the continuous dependence of $u_\sigma$ on $\sigma$ we have $u_\sigma(x,t_1)<s_2$ provided $|\sigma -\sigma_1|\ll 1$. By comparison we have $u_{\sigma}(x,t)< s_2$ for all $x\in \R$ and $t>t_1$. The above trichotomy result then implies that $\omega(u_\sigma)=\{s_1\}$. This shows that $\Sigma_{s_1}$ is a non-empty (since all small $\sigma>0$ with $\sigma \phi<s_2$ belongs to $\Sigma_{s_1}$) open interval $(0, \sigma_{*})$ for some $\sigma_* \in (0,\infty]$. ($\sigma_* = \infty$ is possible, which corresponds to the {\it complete vanishing} phenomena in \cite{LZ}.)

\medskip
\noindent
{\bf Claim 2}. If $\Sigma_1$ is not empty, then it is an open interval $(\sigma^*, \infty)$ for some $\sigma^* \in [\sigma_*,\infty)$.

\noindent
By comparison $\Sigma_1$ is an interval $(\sigma^*, \infty)$ or $[\sigma^*, \infty)$ for some $\sigma^* \geq \sigma_*$. Next, $\Sigma_1$ is open since when $\sigma_2 \in \Sigma_1$, $u_{\sigma_2}(x,t_2)>U_{q_2}(x)$ for  large $t_2$, where $q_2 \in (\theta, 1)$ and $U_{q_2}(x)$ is the stationary solution in Case C in subsection 2.2. By continuous dependence of $u_\sigma$ in $\sigma$, this inequality also holds for other $\sigma$ satisfying $|\sigma -\sigma_2|\ll 1$. This proves Claim 2.

\medskip
Based on these claims we can finish the proof of Theorem \ref{thm1}.  If $\sigma_*=\infty$, then only the small spreading  happens and the proof is completed.  If $\sigma_*<\infty$, then either $\Sigma_1=\emptyset$ and $\Sigma_* = [\sigma_*, \infty)$ or
$\Sigma_1\not=\emptyset$ and $\Sigma_* = [\sigma_*, \sigma^*]$. In any of these cases $\Sigma_*$ is non-empty. Using a similar argument as in proving (ii) and (iii) of \cite[Lemma 6.3]{LZ} (see also \cite{DL,DM} for analogous results for linear-diffusion equations) we conclude that $\Sigma_*$ is a singleton, that is, $\sigma_*=\sigma^*$.

This completes the proof of Theorem \ref{thm1}.
\qed


\section{Asymptotic Behavior of the Solutions in the Range $[0,s_1]$}
In this section we specify $u$ in the range $[0,s_1]$ by using the small sharp wave $Q_{c_s}(x-c_s t)$, and prove Theorem \ref{thm2}. 

\subsection{The former part of Theorem \ref{thm2}} When (i), (ii) or (iii) of Theorem \ref{thm2} holds, the function $H(t)$ is given explicitly. 
The key step is to construct sub- and super-solutions as in \cite{DL, DMZ,DQZ,G}. For this purpose, we need the estimates of $u(r(t)+H(t),t)$ and $u(l(t)-H(t),t)$, which are stated as the following.

\begin{lem}\label{lem:middle-est}
Under the hypotheses {\rm (i), (ii)} or {\rm (iii)} in Theorem \ref{thm2}, there exist $\delta, M', T>0$ such that
\begin{equation}\label{middle-est}
s_1 - M' e^{-\delta t} \leq u(r(t)+H(t),t),\ u(l(t)-H(t),t) \leq s_1 +M' e^{-\delta t},\qquad t\geq T.
\end{equation}
\end{lem}

\noindent
This lemma can be proved as in \cite[Lemma 3.2]{DQZ} and \cite[Lemma 5.1]{G}, we omit the details here. Such a result for reaction diffusion equations (i.e., non-degenerate ones) was  first established in \cite[Lemma 6.5]{DL}, and the analogue for PME was given in \cite[Lemma 3.2]{DQZ}. It is essentially based on the stability of $s_1$.

\medskip
\noindent
{\it Proof of the former part in Theorem \ref{thm2} where (i), (ii) or (iii) holds}. 
We divide the proof into several steps.

\noindent
{\it \underline{Step 1}. Preliminary results on the pressure function $v$}. 
The generalized pressure function $v:= \Lambda(u)$ is defined by \eqref{def-pressure}.
So  \eqref{middle-est} is converted into 
\begin{equation}\label{middle-est-v}
\hat{\varphi}_1 - \hat{M} e^{-\delta t} \leq v(r(t)+H(t),t),\ v(l(t)-H(t),t) \leq \hat{\varphi}_1 +\hat{M} e^{-\delta t},\qquad t\geq T,
\end{equation}
where $\hat{\varphi}_1 := \Lambda(s_1)$ and $\hat{M}$ is some positive number.

Denote $V_{c_s}(x):=\Lambda(Q_{c_s}(x))$ with $V_{c_s}(0)=0$. Then $V_{c_s}(x-c_s t)$ the small 
sharp wave of \eqref{p-v} corresponding to $Q_{c_s}(x-c_s t)$ of \eqref{p}, and    
$$
V_{c_s}(-\infty) = \hat{\varphi}_1,\qquad V'_{c_s}(x)<0 \mbox{ for }x\leq 0.
$$
Based on the definitions of $B, h$ and the properties of $Q_{c_s}$ (cf. subsection 2.1 and section 6), 
we have the following preliminary results:
\begin{enumerate}[(a).]
\item For any small $\varepsilon>0$, there exist $z_0>0$ and some $\varepsilon_0>0$ such that 
$$
\hat{\varphi}_1-\varepsilon\leq V_{c_s}(x)\leq\hat{\varphi}_1 \mbox{  for } x<-z_0,\qquad 
V'_{c_s}(x)\leq -\varepsilon_0 \mbox{ for } -z_0\leq x \leq 0.
$$ 

\item Since $B'(0+0)=A_*>0$, for any small $\varepsilon>0$ and any $\alpha^0>0$, there exist constants $B_1$, $B_2$, $B_3$, $B_4>0$ depending only on $\hat{\varphi}_1$, $\varepsilon$ and $\alpha^0$ such that
\begin{align*}
&B(r)\leq B_1 \mbox{ for } 0\leq r\leq (1+\alpha^0)\hat{\varphi}_1;\qquad \quad 
B(r)\geq B_2  \mbox{ for } \hat{\varphi}_1-\varepsilon\leq r\leq\hat{\varphi}_1, \\
& B_4 \leq B'(r)\leq B_3 \mbox{ for } 0\leq r\leq (1+\alpha^0)\hat{\varphi}_1.
\end{align*}

\item There exists $\hat{\delta}>0$ such that $h'(r)<0$ for $\hat{\varphi}_1-\hat{\delta}\leq r \leq \hat{\varphi}_1+\hat{\delta}$.
So $H(\hat{\delta}):=\inf\limits_{r\in[\hat{\varphi}_1-\hat{\delta},\hat{\varphi}_1+\hat{\delta}]}|h'(r)| >0$. 

\item Since $B\in C^1([0,\infty))$ with $B'(0+0)=A_*$, and $h\in C^1([0,\infty))$ with $h'(0+0)=f'(0)A_*$, for any $\alpha^0>0$, there exists $H_1(\alpha^0)>0$ such that 
$$h(\alpha r)-\frac{\alpha h(r)B(\alpha r)}{B(r)}\leq H_1(\alpha^0)(\alpha-1) \quad \text{for} \ 0<r\leq \hat{\varphi}_1, \ 1\leq \alpha \leq 1+\alpha^0,$$
and for any  $\alpha_0\in(0,1)$, there exists $H_2(\alpha_0)>0$ such that 
$$\frac{\alpha h(r)B(\alpha r)}{B(r)}-h(\alpha r)\leq  H_2(\alpha_0)(1-\alpha) \quad \text{for} \ 0<r\leq \hat{\varphi}_1, \ 1-\alpha_0\leq \alpha \leq 1.$$
\end{enumerate}
These parameters are independent of $v$. In the subsequent construction of sub- and super-solutions, 
additional parameters will be introduced, which may depend functionally on the above established ones.

\noindent
{\it \underline{Step 2}. Construction of a super- and sub-solutions for $v$}. 
We select $\alpha^0_1,\ b^0_1>0$ such that 
$$
v(x,0) \leq (1+\alpha^0_1)V_{c_s}(x-b^0_1) \quad \text{for} \ x\geq 0.
$$
Further, we choose $A_1>0$ large such that the following inequalities hold:
$$
\frac{A_1\delta_1}{\alpha^0_1}>\frac{\hat{\varphi}_1B_3}{B_2},\qquad 
\varepsilon_0c_s\left[\frac{A_1\delta_1}{\alpha^0_1}-\frac{B_3}{B_4}\right]>\hat{\varphi}_1+H_1(\alpha^0) 
\mbox{\ \ \ with\ \ \ } \delta_1:=\min\left\{1, \frac{H(\hat{\delta})}{2}\right\}.
$$
Define  
$$
\overline{V}(x,t):=\alpha_1(t)V_{c_s}(x-c_s t-X^0(t)) \quad \text{for} \ x\in \R, \ t\geq 0,
$$
with
$$
\alpha_1(t):=1+\alpha^0_1e^{-\delta_1t}, \quad X^0(t):= c_s A_1(1-e^{-\delta_1t})+b^0_1,\quad t\geq 0.
$$
We can show that $\overline{V}$ is a supersolution of $v$ in the domain $D:= \{(x,t)\mid x\geq r(t)+H(t),\ t>0\}$.
To do this, we separate the domain $D$ into three parts $D_1 \cup D_2\cup D_3$, where 
$$
D_1 := \{(x,t)\mid c_s t +X^0(t) -z_0\geq x\geq r(t)+H(t),\ t>0\},\quad D_2:= \{(x,t)\mid -z_0\leq x-c_s t -X^0(t) \leq 0,\ t>0\},
$$
and $D_3:= \{(x,t)\mid x-c_s t -X^0(t) \geq 0,\ t>0\}$.  
Then using the asymptotically stability of $\hat{\varphi}_1$ in $D_1$, the fact that $\overline{V}_x$ has negative 
upper bound in $D_2$, as well as the preliminary results (a)-(d) and \eqref{middle-est-v}, 
one can calculate in both domains as in \cite{DQZ, FM} to conclude that $\overline{V}$ is a supersolution in $D$.
(One does not need to worry about the subdomain $D_3$ since both $v$ and $\overline{V}$ take zero in this domain.)

Similarly, we select $\alpha^0_2\in(0,1)$ and $b^0_2\in(0,b)$ satisfying
$$
v(x,0)>(1-\alpha^0_2)V_{c_s}(x-b^0_2) \quad \text{for} \ 0\leq x\leq b^0_2.
$$
Next, we choose $A_2>0$ large enough such that
$$
\frac{A_2\delta_2}{\alpha^0_2}>\frac{\hat{\varphi}_1B_3}{B_2},\quad
\varepsilon_0c_s(1-\alpha^0_2)\left[\frac{A_2\delta_2}{\alpha^0_2}-\frac{B_3}{B_4}\right]>\hat{\varphi}_1+H_2(\alpha_0)
\mbox{\ \ \ with\ \ \ } \delta_2:=\min\left\{1, \delta, \frac{H(\hat{\delta})}{2}  \right\}.
$$
Define   
$$
\underline{V}(x,t):=\alpha_2(t)V_{c_s}(x-c_s t+X_0(t)) \quad \text{for}\ \  x\in \R, \ t\geq 0,
$$
with
$$
\alpha_2(t):=1-\alpha^0_2e^{-\delta_2t}, \quad X_0(t)=c_sA_2(1-e^{-\delta_2t})-b^0_2,\quad t\geq 0.
$$
Then one can verify by the results in (a)-(d) and \eqref{middle-est-v} that $\underline{V}$ is a subsolution of $v$ in the domain $D$.

Combining the super- and sub-solutions together we have  
\begin{equation}\label{lower-upper-est-v}
\underline{V}(x,t)\leq v(x,t)\leq \overline{V}(x,t),\qquad (x,t)\in D,
\end{equation}
and 
\begin{equation}\label{lower-upper-r}
c_s t -X_* \leq c_s t - X_0(t) \leq r(t) \leq c_s t +X^0(t) \leq c_s t +X_*,\qquad t>0,
\end{equation}
for some positive constant $X_*$. 

\noindent
{\it \underline{Step 3}. Convergence of $\tilde{v}$}. Set $\xi := x- c_s t$ and $\tilde{v}(\xi,t) := v(x,t)$. 
Then $\tilde{v}$ solves
\begin{equation}\label{tilde-v-equation}
\tilde{v}_t = B(\tilde{v}) \tilde{v}_{\xi\xi} + \tilde{v}_{\xi}^2 + c_s \tilde{v}_\xi + h(\tilde{v}),\qquad \xi\in \R,\ t>0.
\end{equation}
The estimates in \eqref{lower-upper-est-v} and \eqref{lower-upper-r} are converted into   
\begin{equation}\label{lower-upper-v-xi}
\left\{
\begin{array}{ll}
\alpha_2(t) V_{c_s}(\xi +X_0(t)) \leq \tilde{v}(\xi,t) \leq \alpha_1(t) V_{c_s}(\xi -X^0(t)),&  \xi\geq r(t)+H(t)-c_s t,\ t>0,\\
-X_* \leq \tilde{r}(t):= r(t) -c_s t\leq X_*, &  t>0.
\end{array}
\right.
\end{equation}
Here $\tilde{r}(t)$ is nothing but the right free boundary of $\tilde{v}(\cdot,t)$.

As in \eqref{v-Holder-est}, the function $\tilde{v}(\xi,t)$ has locally uniform H\"{o}lder bound: 
there exist $C>0,\ \alpha_1\in (0,1)$, such that, for any $M>0,\ T>0$ and any increasing time sequence $\{t_n\}$, there holds,
$$
\|\tilde{v}(\xi,t+t_n)\|_{C^{\alpha_1, {\alpha_1}/2} ([-M,M]\times [-T,T])} \leq C.
$$
Hence, for any $\alpha \in (0,\alpha_1)$, there is a subsequence of $\{t_n\}$ (denote it again by $\{t_n\}$) and a function $\tilde{w}_{MT}(\xi,t)$, such that
$$
\|\tilde{v}(\xi,t+t_n )- \tilde{w}_{MT} (\xi,t)\|_{C^{\alpha, \alpha/2} ([-M,M]\times [-T,T])} \to 0 \mbox{\ \ as \ \ }n\to \infty.
$$
Using Cantor's diagonal argument, there exist a subsequence of $\{t_n\}$ (denote it again by $\{t_n\}$) and a function $\tilde{V}(\xi,t)\in C^{\alpha,\alpha/2} (\R\times \R)$ such that
\begin{equation}\label{tilde-v-to-V}
\tilde{v}(\xi, t+ t_n )\to  \tilde{V} (\xi,t) \mbox{\ \ as\ \ }n\to \infty,\quad \mbox{in the topology of } C^{\alpha,\alpha/2}_{loc} (\R^2).
\end{equation}
Furthermore, if $\tilde{V}(\xi,t)>0$ in a domain $\tilde{E}\subset \R^2$, then for any compact subset $\tilde{D}\subset \tilde{E}$, there exists small $\rho>0$ such that
$$
\tilde{V}(\xi,t),\  \tilde{v}(\xi, t+t_n)\geq \rho>0,\quad (\xi,t)\in \tilde{D}, \ n\gg 1.
$$
Then, $\tilde{v}(\xi, t+t_n)$ is classical in $\tilde{D}$, and so, for any $\beta_1\in (0,1)$ and any large $n$, $\|\tilde{v}(\xi, t+t_n )\|_{C^{2+\beta_1, 1+{\beta_1} /2}(\tilde{D})}\leq C$ for some $C$ independent of $n$. This implies that a subsequence of $\{\tilde{v}(\xi, t+t_n )\}$ converges in $C^{2+\beta, 1+\beta/2}(\tilde{D})$ $(0<\beta<\beta_1)$ to the limit $\tilde{V}(\xi,t)$. Thus, $\tilde{V}(\xi,t)$ is a classical solution of 
\eqref{tilde-v-equation} in the domain where $\tilde{V}(\xi,t)>0$.

Using the second conclusion in \eqref{lower-upper-v-xi} and the boundedness of $r'(t)$ in \eqref{r'l'-bound} we see that, there exist a subsequence of $\{t_n\}$ (denote it again by $\{t_n\}$) and a function $\tilde{R}(t)$ such that
\begin{equation}\label{tilde-r-to-R}
\tilde{r}(t+t_n) \to \tilde{R}(t)\quad \mbox{ as }n\to \infty,
\end{equation}
in the topology of $C_{loc}(\R)$. Since $\tilde{V}(\xi,t)$ it is a very weak solution of \eqref{tilde-v-equation} for all $(\xi,t)\in \R^2$,
the Darcy's law corresponding to \eqref{tilde-v-equation} holds: 
$$
\tilde{R}'(t) = - \tilde{V}_x (\tilde{R}(t)-0,t) -c_s,\qquad t\in \R.
$$ 

\noindent
{\it \underline{Step 4}. Convergence of $\tilde{u}$}. Denote by $\tilde{u}(\xi,t)$ and $\tilde{U}(\xi,t)$ respectively the density 
functions corresponding to the pressure functions $\tilde{v}(\xi,t)$ and $\tilde{V}(\xi,t)$. Then 
\begin{equation}\label{tilde-u-equation}
\tilde{u}_t = [A(\tilde{u})]_{\xi\xi} + c_s \tilde{u}_\xi + f(\tilde{u}),\qquad \xi\in \R,\ t>0.
\end{equation}
The first conclusion in \eqref{lower-upper-v-xi} is converted into 
\begin{equation}\label{lower-upper-tilde-u}
(1-\mu_1(t)) Q_{c_s}(\xi +X_0(t)) \leq \tilde{u}(\xi,t) \leq (1+\mu_2(t)) Q_{c_s}(\xi -X^0(t)),\quad  \xi\geq r(t)+H(t)-c_s t,\ t>0,
\end{equation}
for some positive functions $\mu_i(t)\to 0\ (t\to \infty,\ i=1,2)$. The convergence in \eqref{tilde-v-to-V} implies 
\begin{equation}\label{tilde-u-to-W}
\tilde{u}(\xi, t+ t_n) \to \tilde{U}(\xi,t)\quad \mbox{ as }n\to \infty,
\end{equation}
in the sense of \eqref{def-conv}.  As a consequence, 
$$
Q_{c_s} (\xi +X_*) \leq \tilde{U}(\xi,t) \leq Q_{c_s}(\xi -X_*),\qquad |\tilde{R}(t)|\leq X_*,\qquad \xi,\ t\in \R.
$$

Using the intersection number argument to compare $\tilde{u}(\xi,t)$ with $Q_{c_s}(\xi +X_0)$ for each given $X_0\in [-X_*,X_*]$ 
(The intersection number argument for degenerate equations was introduced in \cite{LZ}, which is actually 
applied for pressure functions $\tilde{v}$ and $V_{c_s}$. The conclusions remain obviously valid for $\tilde{u}$ and $\tilde{U}$), we conclude that $\tilde{U}(\xi,t)\equiv Q_{c_s}(\xi-r^*)$ for some $r^*\in [-X_*, X_*]$, and $\tilde{R}(t)\equiv r^*$. In addition, the same intersection number argument even implies that the $\omega$-limit of $\tilde{u}(\cdot,t)$ is unique. In other words, $\tilde{r}(t)\to r^*$ and $\tilde{u}(\xi,t)\to Q_{c_s}(\xi-r^*)$ as $t\to \infty$ (in the sense of \eqref{def-conv}). 
This proves the second limit in \eqref{s-profile}, as well as the first one for $x\in [-M,\infty)$ ($M>0$ is arbitrarily given). 
Combining with the monotonicity (e) in subsection 2.1 we obtain the first limit in \eqref{s-profile}. 

The conclusions in \eqref{s-profile-left} on the left part of the solution, that is, the solution on $(-\infty, -H(t)]$ is proved similarly.
\qed

\subsection{The latter part of Theorem \ref{thm2}}
In this subsection we consider the latter part of Theorem \ref{thm2}, where the big spreading  happens and $c_s =c_z$.
For any $s_*\in (0,s_1)$ and any $s^*\in (s_1,1)$, by the monotonicity of $u(\cdot,t)$ in $[b,r(t)]$, its $s_*$-level set and $s^*$-level set are uniquely defined for large $t$:
\begin{equation}\label{def-level-set}
u(\chi_* (t),t) =s_* ,\qquad u(\chi^* (t),t) =s^*, \qquad t\gg 1.
\end{equation} 
By the big spreading assumption and the monotonicity of $u(\cdot,t)$ in $[b,r(t)]$, they are well-defined for large $t$. 
We will show that they are getting farther and father apart, but with almost the same average speeds: 
$$
d(t):= \chi_*(t)-\chi^*(t) \to \infty,\quad d(t)=o(t)\ \ \ (t\to \infty).
$$ 
Then, it is generally impossible to obtain the exponential decay to $s_1$ as in \eqref{middle-est} for the points between $\chi^*(t)$ and $\chi_*(t)$,
and so the construction of sub- and super-solutions as in the previous subsection is no longer applicable. Instead, we will use an 
idea in Pol\'{a}\v{c}ik \cite{Polacik-MAMS} to proceed our discussion, which is based on 
the phase plane analysis and the zero number argument.

\subsubsection{The estimates of $\chi_*(t)$ and $\chi^*(t)$}
We take a large $T_1>0$, then by the big spreading assumption there exist $M,\delta>0$ such that
\begin{equation}\label{big-spread-middle-0}
\max\{s_2,s^*\} < u(x,t) \leq 1+Me^{-\delta t},\qquad x\in [0,b],\ t\geq T_1,
\end{equation}
and so $u(x,T_1) \geq Q_{c_s}(x)$ for $x\geq 0$. By comparison we have
$$
u(x, t+ T_1) \geq Q_{c_s}(x-c_s t),\qquad x\geq 0,\ t\geq 0.
$$
So, for some $X_1>0$, there holds
\begin{equation}\label{r>ct-0}
\chi_*(t) \geq c_s t -X_1,\qquad t\geq 0.
\end{equation}

On the other hand, using \eqref{big-spread-middle-0} and the method of construction for the supersolution 
as in the previous subsection we have
$$
u(x,t+T_1) \leq (1+Me^{-\delta t}) Q_{c_z} (x-c_z t +X_2),\qquad x\geq 0,\ t>0,
$$
for some $X_2\in \R$. This implies that 
\begin{equation}\label{xi<ct-0}
\chi^* (t) \leq c_z t +X_3 = c_s t+X_3,\qquad t\geq 0,
\end{equation}
for some $X_3>0$.

\begin{lem}\label{lem:dht-to-infty}
There holds $d (t)\to \infty$ as $t\to \infty$.
\end{lem}

\begin{proof}
By contradiction, assume that there exists a time sequence $\{t_n\}$ with $t_n \nearrow \infty\ (n\to \infty)$ 
and $d(t_n)\leq X_4$ for some $X_4>0$. Then, it follows from \eqref{r>ct-0} and \eqref{xi<ct-0} that
$$
-X_1 \leq \chi_*(t_n) -c_s t_n \leq d (t_n) + X_3 \leq X_3 + X_4.
$$
Thus, by taking a subsequence if necessary we may assume that 
$\chi_*(t_n)- c_s t_n \to \tilde{\chi}\in [-X_1, X_3+X_4]\ (n\to \infty)$. 
Set $\tilde{u}(\xi,t):= u(\xi+c_s t,t)$, then $\tilde{\chi}_* (t) := \chi_*(t)- c_s t$ is the rightmost
$s_*$-level set of $\tilde{u}(\cdot,t)$, and
\begin{equation}\label{eq-tilde-u-0}
\tilde{u}_t = [A(\tilde{u})]_{\xi\xi} + c_s \tilde{u}_{\xi} + f(\tilde{u}), \qquad \xi\in \R,\ t>0.
\end{equation}
Using the general convergence result Theorem \ref{thm:general-conv} we conclude that 
$\tilde{u}(\xi,t)$ converges in the topology of $C^2_{loc}((-\infty,\tilde{\chi}])$ and $L^\infty_{loc}(\R)$ 
to some stationary solution $\tilde{w}(\xi)$ of the Cauchy problem of \eqref{eq-tilde-u-0}, that is,
$$
[A(\tilde{w})]_{\xi\xi} + c_s \tilde{w}_{\xi} + f(\tilde{w})=0  \mbox{ for } \xi \in \R,
\qquad \tilde{w}(\tilde{\chi}) =s_*.
$$
In addition, $\tilde{w}(\xi)$ is either positive for all $\xi\in \R$, or its has a unique free boundary 
$\tilde{r}\in (\tilde{\chi},\infty)$ with 
\begin{equation}\label{limit-with-bdry}
\tilde{w}(\xi)=0 \mbox{ for }\xi\geq \tilde{r},\quad \tilde{w}(\xi)>0 \mbox{ for }\xi<\tilde{r},\quad c_s = -[\Lambda(\tilde{w})]_\xi (\tilde{r}-0). 
\end{equation}
The last Darcy's law follows from the fact that $\tilde{w}$ is a stationary solution not only 
of the equation  \eqref{eq-tilde-u-0} but also of its Cauchy problem.

On the other hand, by $0<d (t_n)\leq X_4$ we have $d (t_n) \to \tilde{d} \in [0,X_4]$ (take a subsequence if necessary), 
and so
$$
s^* =  u(\chi^* (t_n),t_n) = \tilde{u} (\chi^* (t_n) -c_s t_n, t_n)  =  \tilde{u}(\chi_*(t_n) -c_s t_n -d(t_n), t_n) \to  
\tilde{w}(\tilde{\chi} -\tilde{d}).
$$
Therefore, $\tilde{w}(x-c_s t)$ is a traveling wave of the equation in \eqref{p}, which 
takes value $s^*\in (s_1,1)$ at some point. According to our classification of the traveling waves in the last section, 
such a traveling wave must be of the Type IV, and so it has a free boundary $\hat{r}$ where 
$[\Lambda(\tilde{w})]_{\xi} (\hat{r})=-\infty$. This, however, contradicts the above stated Darcy's law.
This proves the lemma. 
\end{proof}

We continue to show that
\begin{lem}\label{lem:dht-o}
$d(t) = o(t)$ and $\frac{\chi_* (t)}{t}\to c_s$ as $t\to \infty$.
\end{lem}

\begin{proof}
For any sufficiently small $\varepsilon>0$, by Proposition \ref{prop:last} (iv) in the last section, there exists a traveling wave $U_1(x-(c_s - \varepsilon)t)$ of the equation in \eqref{p}, such that $U_1(x)$ has a compact support:
$$
U_1(x)>0 \mbox{ in } (-L_1, 0),\quad U_1(-L_1)=U_1(0)=0,\quad -[\Lambda(U_1)]' (0-0)\geq c_s -\varepsilon,\quad \max\limits_{[-L_1,0]} U_1(x) \in (s^*,1).
$$
Then when the big spreading happens, we can select $T_2>0$ large such that $u(x,T_2)\geq U_1(x)$. By comparison we have
$$
u(x,t+T_2)\geq U_1 (x-(c_s -\varepsilon) t),\qquad x\in \R,\ t>0.
$$
This implies that
\begin{equation}\label{xi>cst-0}
\chi^* (t) \geq (c_s -\varepsilon) (t-T_2)\qquad \mbox{for}\ t>T_2.
\end{equation}

On the other hand, by Proposition \ref{prop:last} (v) in the last section, there exists a traveling wave 
$U_2(x-(c_s + \varepsilon)t)$ of the equation in \eqref{p}, where $U_2(x)$ satisfies, for some $L_2>0$,
$$
U_2(x)>0 \mbox{ in } (-L_2, 0),\qquad U_2(-L_2)= 1+\varepsilon,\qquad U_2(0)=0,\qquad -[\Lambda(U_2)]' (0-0)\leq c_s +\varepsilon.
$$
Then, when the big spreading  happens, we can select $T_3>T_2$ large such that $u(x,T_3)<U_2(x-X)$ for some $X>0$, 
in their common domain. By comparison we have
$$
u(x,t+T_3)\leq U_2 (x-(c_s +\varepsilon) t-X),\qquad (c_s +\varepsilon)t+X -L_2\leq x\leq (c_s +\varepsilon)t+X,\ t>0,
$$
in their common domain. This implies that
\begin{equation}\label{rt<cst-0}
\chi_* (t) \leq (c_s +\varepsilon) (t-T_3) +X \qquad \mbox{for} \ t>T_3.
\end{equation}

Combining \eqref{xi>cst-0} and \eqref{rt<cst-0} have
$$
d(t) = \chi_* (t)-\chi^* (t) \leq 2\varepsilon t +X, \qquad t>T_3.
$$
Combining \eqref{r>ct-0} and \eqref{rt<cst-0} we have
$$
c_s -\varepsilon \leq\frac{\chi_*(t)}{t}\leq c_s + 2\varepsilon,\qquad t\gg 1.
$$
This proves the lemma.
\end{proof}

\subsubsection{Renormalization and convergence to an entire solution}
For any time sequence $\{t_n\}$ with $t_n\to \infty\ (n\to \infty)$, we consider the sequence 
$v_n(x,t):= v(x+\chi_*(t_n),t+t_n)$. 
Each of them solves the equation \eqref{p-v}. 
Using a similar argument as in Step 3 in the previous subsection, we see that, there exist a subsequence 
of $\{t_n\}$ (denote it again by $\{t_n\}$) and a function $V(x,t)$ such that, as $n\to \infty$, there holds 
\begin{equation}\label{vn-to-V}
v_n (x, t) := v(x+\chi_*(t_n),t+t_n) \to  V (x,t),\quad \mbox{ in the topology of } C^{\alpha,\alpha/2}_{loc}(\R^2),
\end{equation}
for some $\alpha\in (0,1)$. We will prove in this section that $V(x,t)\equiv V_{c_s}(x-c_s t -y_1)$ for some $y_1\in \R$. 
But at the present stage, we even do not know whether $V(\cdot, t)$ has a free boundary or not. 
If it has, we denote it by $R(t)$, otherwise, we regard $R(t)$ as $+\infty$. 
Then the above convergence holds also in the topology of $C^{2+\alpha,1+\alpha/2}_{loc} ((-\infty,R(t))\times \R)$. 
By the definition of $\chi_*(t)$ we have $v_n(0,0)= V(0,0)= \Lambda(s_*)$. 
Thus, for any $t_*<0$, $V(x,t_*)\not\equiv 0$ or $\hat{\varphi}_1 := \Lambda(s_1)$. 
(Otherwise, $V(x,t)\equiv 0$ or $\hat{\varphi}_1$ for all $t\geq t_*$, a contradiction.)  
Furthermore, by the monotonicity of $v(\cdot,t)$ in $[b,r(t)]$ and Lemma \ref{lem:dht-to-infty} 
 we have the following properties for $V$: 
\begin{equation}\label{prop-V-1}
\left\{
\begin{array}{l}
V(0,0)= \Lambda(s_*),\quad V(x,t)\in [0, s^*] \mbox{ for all } x,\ t\in \R,\\
V(x,t)>0 \mbox{ and } V_x (x,t)\leq 0 \mbox{ for } x< R(t),\ t\in \R,\\
R'(t) = - V_x (R(t)-0, t) \mbox{\ \ when\ \ } R(t)<\infty.
\end{array}
\right.
\end{equation}
By the arbitrariness of $s^*$ and using the strong maximum principle we even have 
\begin{equation}\label{prop-V-2}
0<V(x,t)< \hat{\varphi}_1,\qquad V_x(x,t) <0 \quad \mbox{ for } x < R(t),\ t\in \R. 
\end{equation}

\subsubsection{The consistency of $V$ and $V_{c_s}$}
Set 
$$
\xi := x-c_s t,\qquad \hat{v}_n (\xi,t):= v_n(\xi+c_s t,t),\qquad \hat{V}(\xi,t):= V(\xi +c_s t,t),\qquad \xi,t\in \R.
$$
Then $\hat{v}_n\to \hat{V}$ in the $C^{\alpha,\alpha/2}_{loc}(\R^2)$ and 
$C^{2+\alpha,1+\alpha/2}_{loc} ((-\infty,\hat{R}(t))\times \R)$ topologies, for some $\alpha\in (0,1)$
and $\hat{R}(t):= R(t)-c_s t$ when $R(t)<\infty$. All of  $\hat{v}_n$ and $\hat{V}$ are very weak solutions of 
\begin{equation}\label{tilde-v-eq}
v_t = B(v) v_{\xi\xi} + v_{\xi}^2 + c_s v_\xi + h(v),\quad  \xi\in \R.
\end{equation}
The properties \eqref{prop-V-1} and \eqref{prop-V-2} are converted into 
\begin{equation}\label{prop-hat-V}
\left\{
\begin{array}{l}
\hat{V}(0,0)= \Lambda(s_*);\\
0<\hat{V}(\xi,t)< \hat{\varphi}_1,\qquad \hat{V}_\xi (\xi,t) <0 \quad \mbox{ for } \xi < \hat{R}(t),\ t\in \R;\\ 
\hat{R}'(t) = - \hat{V}_\xi  (\hat{R}(t)-0, t) - c_s \mbox{\ \ when\ \ } \hat{R}(t)<\infty.
\end{array}
\right.
\end{equation}
For any given $t_*<0$, we take a $\xi_* <\hat{R}(t_*)$, and denote 
$$
q_* := \hat{V}(\xi_*,t_*) \in (0,\hat{\varphi}_1),\qquad p_* := \hat{V}_\xi (\xi_*,t_*) <0.
$$
In the last section we specify the traveling waves $\varphi(x-ct)$ of \eqref{p-v} by analysis 
on the $\varphi \psi$-phase plane, where $\psi = \varphi'$. In particular, when $c=c_s$, 
we can find a unique traveling wave $\varphi(x-c_s t;c_s, q_*,p_*)$ whose trajectory passes 
through the point $(q_*, p_*)$. We may assume without loss of generality that 
$$
\varphi(\xi_*; c_s,q_*,p_*)=q_*,\qquad \varphi_\xi (\xi_*; c_s,q_*,p_*) =p_*.
$$ 
We will show that $\varphi(\cdot;c_s,q_*,p_*)\equiv V_{c_s}(\cdot)$. But, at this stage, we are 
yet not clear which type among Type I - IV it belongs to.

For any function $w (\xi)$ with $w (\xi)\in (0,\hat{\varphi}_1)$ for $\xi \in I$, we adopt some 
notation in \cite{Polacik-MAMS}:
$$
\tau(w):= \{(w(\xi), w'(\xi))\mid \xi\in I\} \subset S:= \{(\varphi,\psi)\mid 0<\varphi<\hat{\varphi}_1,\ \psi\in \R\}.
$$

\begin{lem}\label{lem:V==Q}
Let $t_*<0$ be any given number. There holds
$$
\hat{V}(\xi,t_*) \equiv V_{c_s} (\xi-y_*) \mbox{ for some } y_* \in \R.
$$
\end{lem}

\begin{proof}
For simplicity, we write $w_1 (\xi):= \hat{V}(\xi,t_*)$ and $q_1(\xi):= \varphi(\xi;c_s,q_*,p_*)$.
(The latter is the traveling wave $\varphi(\xi;c_s,q_*,p_*)$ given as above.)
We first show that 
\begin{equation}\label{tau-w=tau-q}
\tau(w_1) \equiv \tau(q_1)\cap S.
\end{equation}
By the construction of $q_1$, we see that 
\begin{equation}\label{multiple-zero-t*}
\xi_* \mbox{ is a multiple zero of } w_1 - q_1.
\end{equation}

We prove \eqref{tau-w=tau-q} by contradiction. Assume $\tau(w_1) \not\equiv \tau(q_1)\cap S$. Then $\eta:= \hat{V}(\xi,t)-q_1(\xi)$
is a nontrivial solution of a linear equation, and so the zero number argument is applicable. 
(In particular, it is safely applicable near $(\xi,t)=(\xi_*,t_*)$ since the equation is not degenerate
near this point.) By \eqref{multiple-zero-t*} and \cite[Lemma 5.5]{Polacik-MAMS} or \cite[Lemma 2.6]{DM}, 
for all sufficiently large $n$, there exists $s_n\in (-1,1)$ such that 
$$
\hat{v}_n (\xi, s_n +t_*) - q_1(\xi) \mbox{ has a multiple zero } y_n \approx \xi_*.
$$
By the definition of $\hat{v}_n$, this statement is equivalent to the following one: 
\begin{equation}\label{multiple-zero-t}
v(\xi+c_s t_*, s_n + t_n+t_*)-q_1(\xi - z_n) \mbox{ has a multiple zero } y_n \approx 
\xi_* + z_n,
\end{equation}
where $z_n := \chi_*(t_n) +c_s s_n$.

We first assume that $q_1(\xi)$ is of Type II or Type III. For any given large $T>0$, 
since $\chi_*(t_n) + c_s s_n\to \infty\ (n\to \infty)$, we see that for all large fixed $n$,
$v(\xi+c_s t_*,T)$ has no intersection point with $q_1(\xi- z_n)$. According to the profile 
of $q_1$, using the intersection number properties for degenerate equations developed in 
\cite{LZ} we see that for any $t>0$, either $v(\xi+c_s t_*, t+T)$ has no intersection point 
with $q_1(\xi - z_n)$, or they have a unique non-degenerate intersection point. This is 
in particular at time $t= s_n + t_n +t_* - T$, that is, it is true for $v(\xi+c_s t_*, s_n+t_n+t_*)$
and $q_1(\xi -z_n)$. This contradicts \eqref{multiple-zero-t}.

When $q_1(\xi)$ is of Type IV, since $q_1(\xi)>\hat{\varphi}_1$ for $\xi\ll -1$ and 
$q_1$ has a free boundary $r_1$ on the right where $q'_1(r_1-0)=-\infty$. The intersection
number argument as above remains valid, and so we derive a contradiction again. 

Finally, we assume that $q_1\equiv V_{c_s}$ is the small sharp wave. Fixed a large $T>0$ such that
(according to the big spreading assumption) $v(\xi+c_s t_*, T)$ goes down across $\hat{\varphi}_1$
strictly. In other words, there exist $\epsilon_1, \epsilon_2>0$ such that 
$$
\hat{\varphi}_1 + \epsilon_1 = v(\xi_1 + c_s t_*,T),\qquad 
\hat{\varphi}_1 - \epsilon_1 = v(\xi_2 + c_s t_*,T),\qquad 
v_\xi (\xi + c_s t_*,T) <-\epsilon_2 \mbox{ for } \xi\in [\xi_1, \xi_2].
$$ 
Then we can take $n$ sufficiently large such that $v(\xi+c_s t_*, T)$ has exactly one non-degenerate
intersection point in $[\xi_1, \xi_2]$ with $q_1 (\xi- z_n)$. Both $v(\xi+c_s t_*, t+T)$
and $q_1 (\xi- z_n)$ are solutions of \eqref{tilde-v-eq}, and both satisfy the same 
Darcy's law as in \eqref{prop-hat-V}. Moreover, for any $t$, they are not exactly the same. 
(Otherwise, $w_1(\xi+c_s t_*)\equiv \hat{V}(\xi+c_s t_*, t_*)\equiv q_1(\xi -z_n)$, 
and so $\tau({w_1})=\tau(q_1)$.) 
Therefore, the intersection number argument is applied and so 
$v(\xi+c_s t_*, s_n + t_n+t_*)$ has no degenerate intersection point with $q_1 (\xi- z_n)$, contradicts 
\eqref{multiple-zero-t} again. 

We have proved \eqref{tau-w=tau-q}. By \eqref{prop-hat-V} we see that $w_1$ is monotonically 
decreasing and with values in $(0,\hat{\varphi}_1)$. Therefore, $q_1$ can not be of Type II, III, or IV.
It must be of Type I with suitable shift.

This proves the lemma. 
\end{proof}

\subsubsection{Proof of the latter part of Theorem \ref{thm2}}
By Lemma \ref{lem:V==Q}, for any $t_*<0$, there holds $\hat{V}(\xi,t_*)\equiv V_{c_s}(\xi-y_*)$. 
Since both $\hat{V}(\xi, t+t_*)$ and $V_{c_s}(\xi-y_*)$ solutions of the Cauchy problem 
of \eqref{tilde-v-eq}, and the latter is a stationary one, we have   
$\hat{V}(\xi, t+t_*) \equiv V_{c_s}(\xi-y_*)$ for all $t\geq 0$. In particular, 
at $(\xi,t)=(0,-t_*)$ we have $V_{c_s}(-y_*) = \Lambda(s_*)$, that is, $-y_*$ is nothing but the 
$\Lambda(s_*)$-level set of $V_{c_s}$, and so is uniquely defined. By the 
arbitrariness of $t_*<0$ and the uniqueness of $y_*$, we actually have
$$
v(\xi+\chi_*(s), s) \to V_{c_s}(\xi-y_*) \mbox{\ \ as\ \ }s\to \infty,
$$
in the $C^{\alpha}_{loc}(\R)$ topology. In particular, due to 
$$
v(y_*+\chi_*(s), s) \to V_{c_s}(0)=0 \equiv v(r(s),s),
$$
we have $r(s)-\chi_*(s)\to y_*\ (s\to \infty)$.  Together with the conclusion in Lemma 
\ref{lem:dht-o} we have $r(s)=c_s s + o(s)$ as $s\to \infty$. 
In addition, 
\begin{eqnarray*}
v(\xi + r(s),s) & = & [v(\xi +\chi_*(s) +(r(s)-\chi_*(s)),s) - v(\xi+\chi_*(s) + y_*,s)] \\
& &  + v(\xi+\chi_*(s) + y_*,s) \to V_{c_s}(\xi)\quad \mbox{\ \ as\ \ }s\to \infty,
\end{eqnarray*}
in the $C^{\alpha}_{loc}(\R)$ topology.  

This proves the latter part of Theorem \ref{thm2}. \qed

\section{Asymptotic Behavior of the Solution in the Range $[s_1,1]$}

In this section, we give the convergence of $u$ in the range $[s_1,1]$ when big spreading happens.

\medskip
\noindent
{\it Proof of Theorem \ref{thm3}}. 
(i) When $c_s < c_z$, we can use a similar argument as proving Lemma 
\ref{lem:middle-est} (see, also \cite[Lemma 6.5]{DL}, or \cite[Lemma 3.2]{DQZ}, \cite[Lemma 5.1]{G}) to show that, 
for any $\tilde{c}\in(0,c_s)$, there exist $\delta, M, T>0$ such that
\begin{equation}\label{top-est}
\left\{
\begin{array}{l}
 u(x,t)\leq 1+Me^{-\delta t} \text{ for all } x\geq 0 \text{ and } t\geq T, \\
 u(x,t)\geq 1-Me^{-\delta t} \text{ for all } x\in[0,\tilde{c}t] \text{ and } t\geq T.
\end{array}
\right.
\end{equation}
With this decay rate in the middle part, we can use the same idea as in Step 2 in 
subsection 4.1 to construct super- and sub-solutions by the big sharp wave $Q_{c_b} (x-c_b t)$. 
(This is possible essentially due to the asymptotic stability of $1$ and the fact 
that $(Q_{c_b})_{x}$ has negative upper bound near its front.) The conclusion (i) 
is then proved as in subsection 4.1.

(ii) When $c_s >c_z$, the big spreading solution $u$ is actually characterized by a propagating terrace. 
In the range $[0,s_1]$, the asymptotic behavior of $u$ has been specified 
in Theorem 1.2 (iii). We now consider the range $[s_1, 1]$. 
Using the traveling wave $Q_{c_z}(x-c_z t)$, one can construct sub- and super-solutions 
as in \cite{FM} (see also subsection 4.1), and then prove the conclusion (ii)
in a similar way as in the previous section.
(Note that, the equation is non-degenerate in the range $[s_1,1]$.)

(iii) Let us now consider the critical case: $c_s =c_z$. As we have seen in subsection 
4.2, the approach in this case is quite different and complicated. 
Fix $s_*\in (0,s_1)$ and $s^*\in (s_1,1)\backslash \{s_2\}$, recall that we use $\chi_*(t)$ and $\chi^*(t)$ to 
denote the right $s_*$-level set and $s^*$-level set of $u(\cdot,t)$, respectively. 
For any time sequence $\{t_n\}$ tending to infinity, by parabolic estimates we may assume that 
(by taking a subsequence if necessary)  
\begin{equation}\label{un-to-U-iii}
 u(x+\chi^* (t_n), t+t_n) \to U(x,t)\mbox{\ \ as\ \ }n\to \infty,
\end{equation}
in the topology of $C^{2+\alpha,1+\alpha/2}_{loc}(\R^2)$, for some $\alpha\in (0,1)$ and entire solution $U$.
By the monotonicity of $u(\cdot,t)$ in $[b,r(t)]$, Lemmas \ref{lem:dht-to-infty} and \ref{lem:dht-o} 
we see that $U$ satisfies  
$$
U(0,0)=s^*,\quad U(-\infty,t)\leq 1,\quad U(+\infty,t) \geq s_*,\quad U_x(x,t)<0 \mbox{ for all }x,t\in \R.
$$
Since $s_*\in (0,s_1)$ is arbitrarily chosen we even have $U(+\infty, t)\geq s_1$.

Using the moving frame $\xi:= x-c_z t$,  we denote 
$$
\hat{u}_n (\xi,t):= u(\xi + c_z t +\chi^* (t_n), t+t_n)|_{E_n},\qquad \hat{U}(\xi,t):= U(\xi+c_z t,t),
$$
where $E_n:= \{(\xi,t) \mid 0 \leq \xi +c_z t +\chi^*(t_n) \leq \chi_*(t+t_n),\ t>-t_n\}$.
Then all of $\hat{u}_n$ and $\hat{U}$ are solutions of $u_t = [A(u)]_{\xi\xi} + c_z u_{\xi} + f(u)$. 
Note that all the stationary solutions of this equation correspond to traveling waves  
of the equation in \eqref{p}. They are divided into four types: those with increasing profiles 
in $[s_1,1]$; those with oscillation profiles in $[s_1,1]$; 
the traveling wave $Q_{c_z}$; monotonically decreasing ones with values bigger than
$1$ on the left side. Using the same approach as in subsections 4.2.3 and 4.2.4 
(see also section 6 in \cite{Polacik-MAMS}) we conclude that $\hat{U}$ must coincide with one 
of them. Combining together with the properties of $U$ we derive  
$$
\hat{U}(\xi,t)\equiv Q_{c_z} (\xi-y^*), \qquad \mbox{ where } y^* \mbox{ satisfies } Q_{c_z}(-y^*)=s^*.
$$
By the uniqueness of $y^*$ and $Q_{c_z}$, the convergence 
\eqref{un-to-U-iii} can be improved as  
\begin{equation}\label{un-to-U-iii-1}
u(x +\chi^* (s), t+s) \to U(x,t) \equiv \hat{U}(x-c_z t,t) \equiv Q_{c_z}(x-c_z t -y^*)\quad \mbox{\ \ as\ \ }s\to \infty,
\end{equation}
in the topology of $C^{2+\alpha,1+\alpha/2}_{loc}(\R^2)$. In particular, 
$$
u_x (\chi^*(s),s)\to Q'_{c_z} (-y^*) <0,\qquad u_t (\chi^*(s),s)\to -c_z Q'_{c_z}(-y^*)\quad \mbox{\ \ as\ \ }t\to \infty.
$$
Differentiating $u(\chi^*(s),s)\equiv s^*$ in $s$ we have 
$$
(\chi^*)'(s) = -\frac{u_t (\chi^*(s),s)}{ u_x (\chi^*(s),s)} \to c_z \quad \mbox{\ \ as \ \ }s\to \infty.
$$
Denote $\theta_+(s):= \chi^*(s) -c_z s +y^*$, we see that $\theta'_+(s)\to 0\ (s\to \infty)$, and 
$$
u(x + c_z s +\theta_+(s), s)  =  u(x+y^*+\chi^*(s),s) \to Q_{c_z}(x)\quad \mbox{\ \ as\ \ }s\to \infty,
$$
in the topology of $C^2_{loc}(\R)$. This proves the first limit in \eqref{cs=cz=upper}. The second one 
is proved similarly.

This completes the proof of Theorem \ref{thm3}. \qed

\section{Stationary Solutions and Traveling Waves}\label{sec:sstw}
As a supplement, in this section we provide detailed construction for stationary and traveling wave solutions of the equation, by using the phase plane analysis. One will see that our construction is much more complicated than the cases for reaction diffusion equations and that for porous media equations. The difficulties are mainly caused by the general degenerate term $A(u)$.

\subsection{Positive stationary solutions}
In this subsection, we apply the phase plane analysis to give the positive stationary solution $u=q(x)$ of the equation in \eqref{p}.
We rewrite the equation $[A(q)]''+f(q)=0$ in the equivalent form:
\begin{equation}\label{tt-two}
	\begin{cases}
		p=[A(q)]'=A'(q)q',\\
		p'=-f(q),
	\end{cases}
\end{equation}
or a first order equation
$$
\frac{dp}{dq}=-\frac{A'(q) f(q)}{p}.
$$
Its first integral is
\begin{align}\label{shoucijifen}
	p^{2}=C-2\int_{0}^{q} A'(r) f(r)dr,
\end{align}
for suitable $C$. We consider only the trajectories in the region $\{q\geq0\}$ which correspond to nonnegative solutions. \\
\textbf{Case~A.}~Constant solutions $U_q(x)$. In the $qp$-plane, the points $(0, 0)$, $(s_{1},0)$, $(s_{2},0)$, $(1, 0)$ are singular ones. They correspond to constant stationary solutions $U_{q}(x)\equiv{q}$ for $q$~=~0, $s_{1}$, $s_{2}$, 1.

\noindent
\textbf{Case~B.}~Ground state solution $U^*(x)$. We consider the trajectory corresponding to $C=C_* :=2\int_0^{s_1} A'(r) f(r)dr$. Its formula is
\begin{equation}\label{homo-clinic}
p^{2}=-2\int_{s_1}^{q} A'(r) f(r)dr.
\end{equation}
It is a homoclinic trajectory starting at ($s_{1}$, 0), crossing ($\theta$, 0) and then tending to $(s_1,0)$, where $\theta\in{(s_{2}, 1)}$ is given by
$$
\int_{s_{1}}^{\theta} A'(r) f(r)dr=0.
$$
Assume it is normalized by $q(0)=\theta,\; q'(0)=0$, and the support of $q(x)-s_1$ is $[-L,L]$ for some~$L\in{(0,\infty]}$. We now show that $L=\infty$ under the assumption $f'(s_1)<0$.

When $x$ is close to $-L$, we have $s_1<q(x)< s_1 +\delta$ for any sufficiently small $\delta>0$. So
$$
A'(s_1)q'(x) \leq A'(q)q'(x) = p (x) = \left( -2\int_{s_1}^{q} A'(r) f(r)dr \right)^{1/2} < \kappa (q-s_1),
$$
where $\kappa = 2\sqrt{-A'(s_1)f'(s_1)}$. Integrating the first equation in \eqref{tt-two} we have
\begin{eqnarray*}
L & = & \int_{-L}^0 dx = \int_{s_1}^\theta \frac{A'(q) dq}{p(q)} =
\int_{s_1}^{s_1 +\delta} \frac{A'(q) dq}{p(q)} + \int_{s_1 +\delta}^\theta \frac{A'(q) dq}{p(q)} \\
& \geq & \int_{s_1}^{s_1 +\delta} \frac{A'(s_1) dq}{\kappa (q-s_1)} + \int_{s_1 +\delta}^\theta \frac{A'(q) dq}{p(q)} =\infty.
\end{eqnarray*}
This means that $U^*(x)>0$ for all $x\in \R$. It is the analogue of the ground state solution in bistable reaction diffusion equations, and is denoted by $U^*(x)$.

\noindent
\textbf{Case~C.} Solution with compact supports. When we take $0<{C}<C_*$ in the first integral,~the trajectory~$\Gamma_{1}$:~$p^{2}+2\int_{0}^{q}  A'(r) f(r)dr=C$~is a connected curve, symmetric with respect to the $q$-axis, which connects the points $(0,\sqrt{C})$,~$({q_{1}},0)$~and $(0,-\sqrt{C})$ for a unique ${q_{1}}\in{(0,s_{1})}$.
The trajectory corresponds to a stationary solution $U_{q_{1}}(x)$, under the normalized condition $U_{q_{1}}'(0)=0$, it satisfies
$$
U_{q_1}(0)=q_1,\quad U_{q_1}(\pm L_1)= 0, \quad U'_{q_1}(x)<0 \quad\text{for}~x\in{(0,L_{1})}.
$$
In addition, by \eqref{shoucijifen} we have
$$
[A(U_{q_1})] '(x)= - \Big(C-2\int_{0}^{U_{q_1}} A'(r) f(r) dr\Big)^{1/2} \to -{\sqrt{C}},\mbox{\ \ as\ } x\to L_1-0.
$$
Using \eqref{shoucijifen} again, it is not difficult to show the following result.

\begin{lem}\label{bdlkd}
If $f'(0)>0$, then $L_1\to0$ as $q_1\to0$.
\end{lem}

%
%

Similarly, when we take $C_*<{C}<C^*:=2\int_{0}^{1}  A'(r) f(r)dr$,~the trajectory corresponds to a stationary solution $U_{q_{2}}(x)$, which satisfies
$$
U_{q_{2}}(0)=q_{2},\;\;\; U'_{q_2}(0)=0,\quad  U_{q_{2}}(\pm{L_{2}})=0,~L_{2}>0,\;\;\; U_{q_{2}}'(x)<0\quad\text{for}~x\in{(0,L_{2})}.
$$

\noindent
\textbf{Case~D.}~Monotonic solution. When $C=C_*$, besides the homoclinic orbit, there are two other trajectories, whose functions are
$$p={P_{1}(q):={\Big(2\int_{q}^{s_{1}}  A'(r) f(r) dr}\Big)}^{1/2},~p=-{P_{1}(q)},~0\leq{q}<{s_{1}}.$$
They connect regular points $(0,\pm{\sqrt{C_{1}}})$ with the singular point $(s_{1},0)$. Denote the corresponding stationary solutions by $U_1^+(x)$ and $U_1^-(x):= U_1^+(-x)$, where
$$
U_{1}^{+}(0)=0,\;\;\; U_{1}^{+}(x)>0~\text{for}~x>0,\;\;\; \text{and}
~U_{1}^{+}(x)\nearrow{s_{1}}~\text{as}~x\rightarrow{+\infty}.
$$

When $C=C^*$, there are two trajectories which correspond to two solutions $U_2^+(x)$ and $U_2^-(x):= U_2^+(-x)$, where
$$
U_{2}^{+}(0)=0,\;\;\; U_{2}^{+}(x)>0~\text{for}~x>0,\;\;\; \text{and}~U_{2}^{+}(x)\nearrow{1}~\text{as}~x\rightarrow{+\infty}.
$$

\noindent \textbf{Case~E.}~Periodic solution. In the $qp$-phase plane, the singular point $(s_{2},0)$~is a center.~There are infinitely many closed trajectories surrounding it. Each of them corresponds to a positive stationary periodic solution $U_{per}(x)$.

\subsection{Traveling Waves}
We now use the phase plane analysis to investigate the traveling waves of the equation in \eqref{p}.
It is convenient to use the generalized pressure equation of $v$:
\begin{equation}\label{vtwfc}
v_t = B(v)v_{xx} + v_x^2 + h(v),
\end{equation}
where $B$ and $h$ are given by \eqref{u-to-v}. It is obvious that $B,h\in C^2((0,\infty))$. On their smoothness at $v=0$ we have the following lemma.

\begin{lem}\label{lem:B-h}
If we make the supplementary definition: $\Lambda (0)=\lambda (0)=0$ and $h(0)=0$, then the properties in \eqref{prop-B-h} hold.
\end{lem}

\begin{proof}
Differentiating $\Lambda(\lambda(v))\equiv v$ twice we have
\begin{equation}\label{Lambda-lambda}
\Lambda'(\lambda(v))\lambda'(v)\equiv 1,\qquad \Lambda''(\lambda(v))[\lambda'(v)]^2 = -\Lambda'(\lambda(v))\lambda''(v),\quad v>0.
\end{equation}
By the definition \eqref{def-pressure} we have
\begin{equation}\label{Lambda'}
u \Lambda'(u) = A'(u),\qquad \Lambda'(u) + u\Lambda''(u) = A''(u),\quad u>0.
\end{equation}
So $B(0)=A'(\lambda(0))=A'(0)=0$, and for $v>0$,
$$
B(v) := A'(\lambda(v)) = \lambda(v) \Lambda'(\lambda(v)) = \frac{\lambda(v)}{\lambda'(v)}.
$$
Thus
$$
B'(0+0)  =  \lim\limits_{v\to 0+0} \frac{A'(\lambda(v))}{v} = \lim\limits_{u\to 0+0} \frac{A'(u)}{\Lambda(u)} =  \lim\limits_{u\to 0+0} \frac{A''(u)}{\Lambda'(u)}
=  \lim\limits_{u\to 0+0} \frac{u A''(u)}{A'(u)} = A_*,
$$
by the assumption in \eqref{ass-A}. On the other hand, when $v>0$ we have
$$
B'(v)=A''(\lambda(v))\lambda'(v) = \frac{A''(\lambda(v))}{\Lambda'(\lambda(v))} = \frac{\lambda(v)A''(\lambda(v))}{A'(\lambda(v))}\to A_*\mbox{\ \ as\ \ } v\to 0+0,
$$
and so $B\in C^1([0,\infty))$.

Denote $a_0:= f'(0)$. We now study $h(v)$. First,
$$
\lim\limits_{v\to 0+0} h(v) = \lim\limits_{v\to 0+0} \frac{f(\lambda(v))}{\lambda'(v)} =
\lim\limits_{v\to 0+0} a_0 \lambda(v) \Lambda'(\lambda(v)) = a_0 \lim\limits_{v\to 0+0} A'(\lambda(v))=0.
$$
Next we have
\begin{eqnarray*}
h'(0+0) & = & \lim\limits_{v\to 0+0} \frac{f(\lambda(v))}{v\lambda'(v)} = \lim\limits_{v\to 0+0}  \frac{a_0 A'(\lambda(v))}{v} =  a_0 \lim\limits_{u\to 0+0} \frac{A'(u)}{\Lambda(u)} \\
& = & a_0 \lim\limits_{u\to 0+0} \frac{A''(u)}{\Lambda'(u)}
= a_0 \lim\limits_{u\to 0+0} \frac{u A''(u)}{A'(u)} = a_0 A_*.
\end{eqnarray*}
Finally, with $u=\lambda(v)$ we have
\begin{eqnarray*}
h'(v) & = & f'(\lambda(v)) - f(\lambda(v))\frac{\lambda''(v)}{[\lambda'(v)]^2} =
f'(u) + \frac{f(u)}{u} \cdot \frac{u \Lambda''(u)}{\Lambda'(u)} \\
& =  & f'(u) + \frac{f(u)}{u} \cdot \frac{ A''(u) -\Lambda'(u)}{\Lambda'(u)}
 =   f'(u) - \frac{f(u)}{u} + \frac{f(u)}{u} \cdot \frac{ u A''(u)}{A'(u)},
\end{eqnarray*}
so $ h'(v) \to a_0A_*$ as $v\to 0+0$. This proves $h\in C^1([0,\infty))$.
\end{proof}

For any $c>0$, we consider the traveling wave $v(x,t)=\varphi (x-ct)$, then the pair $(\varphi (\zeta),c) $ solves
\begin{align}\label{vtwq}
B(\varphi) \varphi_{\zeta\zeta}+\varphi_{\zeta}^{2}+c\varphi_{\zeta}+h(\varphi)=0,
\end{align}
For the purpose of studying the phase portrait, we rewrite equation \eqref{vtwq} as a pair of first order equations:
\begin{equation}\label{qptw}
\begin{cases}
	\varphi'= \psi,  \\
	\psi'=\displaystyle{\frac{-c\psi - \psi^2- h(\varphi)}{B(\varphi)}},
\end{cases}
\end{equation}\\
where $'$ denotes $d/d\zeta$. To eliminate the singularity at $\varphi=0$, introduce the new independent variable $\xi$ defined by the symbolic equation $\frac{d\zeta}{d\xi}= B(\varphi)$. Then the system \eqref{qptw} becomes
\begin{equation}\label{qptwb}
\begin{cases}
	\dot \varphi =B(\varphi) \psi,   \\
	\dot \psi =-\psi^{2}-c\psi -h(\varphi),
\end{cases}
\end{equation}
where $\cdot$ denotes $d/d\xi$. Note that \eqref{qptwb} is not singular and equivalent to \eqref{qptw} for $\varphi>0$ and has the same trajectories in the positive half plane $H=\{(\varphi,\psi)~|~\varphi>0,-\infty<\psi<+\infty\}$.
Set
$$
\hat{\varphi}_0 = \hat{\varphi}_4 = 0,\quad \hat{\varphi}_i = \Lambda (s_i) := \int_0^{s_i} \frac{A'(r)}{r} dr\ \ (i=1,2,3) \mbox{\ \ with\ } s_3 :=1.,
$$
and
$$
\hat{\psi}_0 = \hat{\psi}_1 = \hat{\psi}_2 = \hat{\psi}_3 = 0,\quad \hat{\psi}_4 := -c.
$$
Then the system \eqref{qptwb} has exactly 5 singular points: $R_i(\hat{\varphi}_i, \hat{\psi}_i)\ (i=0,1,2,3,4)$.  Linearizing the system \eqref{qptwb} at $R_i$ we obtain the associate Jacobian
$$
J(R_i)=
\begin{pmatrix}
	A_* \hat{\psi}_i  &  \hat{b}_i  \\
	- \hat{a}_i  & -2 \hat{\psi}_i-c
\end{pmatrix}
$$
where
$$
\hat{b}_i := B(\hat{\varphi}_i)\ \ (i=0,1,2,3,4),\quad \mbox{ and so }\ \ \hat{b}_0 = \hat{b}_4 =0,\ \ \hat{b}_1, \hat{b}_2, \hat{b}_3>0,
$$
$$
\hat{a}_0 = \hat{a}_4 := a_0 A_*,\quad \hat{a}_i := h'(\hat{\varphi}_i) = f'(s_i)\ \ (i=1,2,3),\quad
\mbox{ and so } \ \ \hat{a}_1,  \hat{a}_3<0,\  \hat{a}_2>0.
$$
Therefore, we have the following conclusions:
\begin{enumerate}[$(1)$]
	\item $R_0$ is topologically a node in the right $\varphi\psi$-plane by using the method employed in \cite{ALGM}, and $J(R_0)$ has the eigenvalues $0$ and $-c$ with the corresponding eigenvectors $(c,-a_0 A_* )$ and $(0,1)$, respectively;

	\item $R_1$ is saddle, and $J(R_1)$ has the eigenvalues $\lambda_1^\pm (c)=(-c\pm \sqrt{c^2-4 \hat{a}_1 \hat{b}_1} ) /2$ with the corresponding eigenvectors $(\hat{b}_1, \lambda_1^\pm(c))$, and $\lambda^+_1(c)$ is strictly decreasing in $c>0$;
	
\item $J(R_2)$ has the eigenvalues $\lambda_2^\pm(c)=(-c\pm \sqrt{c^2-4 \hat{a}_2 \hat{b}_2 })/2$, so there are three different cases: if $c^2<4 \hat{a}_2 \hat{b}_2 $, $R_2$ is a focus; if $c^2=4\hat{a}_2 \hat{b}_2 $, $R_2$ is an unidirectional node; if $c^2>4 \hat{a}_2 \hat{b}_2 $, $R_2$ is a bidirectional node.
	
\item $R_3$ is saddle, and $J(R_3)$ has the eigenvalues $\lambda_3^\pm(c)=(-c\pm\sqrt{c^2- 4\hat{a}_3 \hat{b}_3 })/2$ with the corresponding eigenvectors $(\hat{b}_3,\lambda_3^\pm(c))$, and $\lambda^+_3(c)$ is strictly decreasing in $c>0$;

	\item $R_4$ is a saddle, and $J(R_4)$ has the eigenvalues $c$ and $-c A_*$ with the corresponding eigenvectors $(0,1)$ and $(c(A_*+1), a_0 A_*)$, respectively.
\end{enumerate}

Now we study the trajectory $\Gamma^c$ starting from $R_1$ and entering the domain $D:= \{(\varphi,\psi) \mid 0\leq \varphi \leq \hat{\varphi}_1,\ \psi <0\}$.
Denote the corresponding function by $\psi= \psi(\varphi;c)$, and denote the contact
point between $\Gamma^c$ and the negative $\psi$-axis by $P(0,\psi^c)$, if it exists.
We have the following results.

\begin{lem}\label{lem:3-TW}
For any $c'> c\geq 0$, the following conclusions hold.
\begin{enumerate}[{\rm (i)}]
\item $\Gamma^{c'}$ lies above $\Gamma^c$ as long as they stay in $D$;
\item When $c>0$ is sufficiently small, $\Gamma^c$ either goes down to $-\infty$ along the right side of the $\psi$-axis, or it contacts the negative $\psi$-axis at point $P(0,\psi^c)$ with  $\psi^c<-c$;
\item When $c>0$ is sufficiently large, $\Gamma^c$ either tends to $R_0$ from the right-down side, or it contacts the negative $\psi$-axis at point $P(0,\psi^c)$ with  $\psi^c > -c$;
\item $\psi^c$ is strictly increasing in $c$ as long as $\psi^c\in (-\infty,0)$.
\end{enumerate}
\end{lem}

\begin{proof}
(i). Note that $(\hat{b}_1, \lambda^+_1(c))$ is the direction along which $\Gamma^c$ leaving the singular point $R_1$. Since $\lambda^+_1(c)$ is strictly decreasing in $c\geq 0$, we see that $\Gamma^{c'}$ lies above $\Gamma^c$ when they are near $R_1$. Then, both of them extend leftward, it is easy to show by contradiction that $\Gamma^{c'}$ remains above $\Gamma^c$ as long as they stay in $D$.

(ii). Denote $I_1:= (0,\frac12 \hat{\varphi}_1]$ and $I_2 := [\frac12 \hat{\varphi}_1, \hat{\varphi}_1)$. By our assumption on $f$ we see that, there exists $k>0$ such that
$$
h(\varphi) \geq -k (\varphi - \hat{\varphi}_1),\qquad \varphi \in I_2.
$$
Set $B^0:= \sup\limits_{r\in I_1\cup I_2} B(r)$. We take $\rho >0$ and $c\geq 0$ with
\begin{equation}\label{cond-rho-c}
2 B^0 \rho^2 < k,\quad  \rho < \lambda^+_1 (1),\quad 0\leq c <\min\left\{1, \frac{k}{2\rho}, \frac{\rho}{2} \hat{\varphi}_1 \right\}.
\end{equation}
Then, in a small neighborhood of $R_1$, $\Gamma^c$ lies below $\Gamma^1$ since $c<1$, and so it is  also below the straight line $\ell_2:\ \psi = \rho(\varphi -\hat{\varphi}_1)$ in this neighborhood.

We now show that this remains hold for all $\varphi\in I_2$. By contradiction, we assume $\Gamma^c$ contacts $\ell_2$ at some points in the region $\varphi\in I_2$. Denote the rightmost one of such points by $P(\varphi_0, \psi_0)$. Then at $\varphi=\varphi_0$ we have
$$
B^0 \rho \geq B(\varphi_0) \rho \geq  B(\varphi_0) \frac{d\psi(\varphi_0;c)}{d\varphi} = -\psi_0 -c - \frac{h(\varphi_0)}{\psi_0} \geq -c + \frac{k}{\rho} > \frac{k}{2\rho},
$$
contradict the first inequality in \eqref{cond-rho-c}.
As a consequence, we have
$$
\psi \Big( \frac12 \hat{\varphi}_1;c \Big) < \psi^* := - \frac12 \rho \hat{\varphi}_1 ,
$$
and so $\Gamma^c$ lies below the horizontal line $\ell_1:\ \psi= \psi^*$ when $\varphi$ is near $\frac12 \hat{\varphi}_1$. A simple discussion as above shows that this also remains true for all $\varphi\in I_1$. Consequently, when $c>0$ satisfies the conditions in \eqref{cond-rho-c} the conclusions in (ii) hold.

(iii). For any given $\rho_1>0$, denote by $\ell_3$ the straight line $\psi= \rho_1 (\varphi -\hat{\varphi}_1)$. We set $D_1:= \{(\varphi,\psi) \mid 0<\varphi<\hat{\varphi}_1, \rho_1 (\varphi -\hat{\varphi}_1) \leq \psi <0\}$, and show that, when $c>0$ is sufficiently large, $\Gamma^c$ remains in $D_1$. Clearly, this is true when $\Gamma^c$ is near $R_1$ since $\lambda^+_1(c)$ is small when  $c$ is large. By contradiction, if $\Gamma^c$ escapes from $D_1$ for the first time from some point $P(\varphi_1, \psi_1)\in \ell_3$, then
$$
B(\varphi_1) \rho_1 \leq B(\varphi_1) \frac{d\psi(\varphi_1;c)}{d\varphi} = -\psi_1 - c -\frac{h(\varphi_1)}{\psi_1}.
$$
This is impossible when $c>0$ is sufficiently large. Hence, when $c$ is large we have $\psi^c\geq -\rho_1 \hat{\varphi}_1 > -c$.

(iv). The conclusion follows from the previous ones.
\end{proof}

Based on this lemma we give the main conclusions in this subsection.

\begin{prop}\label{prop:last}
Assume \eqref{ass-A} and \eqref{ass-f}. Then we have the following traveling waves.
\begin{enumerate}[{\rm (i)}]
\item There exists $c_z>0$ such that the equation in \eqref{p} has a traveling wave solution $Q_{c_z}(x-c_z t)$ connecting $1$ and $s_1$, as stated in \eqref{trav-cond};

\item There exists  $c_s>0$ such that the equation in \eqref{p} has a unique small sharp wave  $Q_{c_s}(x-c_s t)$ connecting $s_1$ and $0$, as stated in \eqref{sharp-wave};

\item When $c_z > c_s$, there exists $c_b \in (c_s, c_z)$ such that the equation in \eqref{p} has a unique big sharp wave $Q_{c_b}(x-c_b t)$ connecting $1$ and $0$, as stated in \eqref{sharp-wave};

\item When $c_s= c_z$, for any $c'\in (0,c_s)$ and any $d$ with $0<1-d\ll 1$, the equation in \eqref{p} has a traveling wave solution $U_1(x-c' t)$ such that $U_1(x)$ has support $[-L_1, 0]$, $\max\limits_{-L_1\leq x\leq 0}U_1(x)=d$ and $[\Lambda(U_1)]'(0-0) < -c'$;

\item When $c_s =c_z$, for any $c''>c_s$, the equation in \eqref{p} has a traveling wave solution $U_2(x-c'' t)$ such that $U_2(x)>0$ in $(-L_2, 0)$ for some $L_2 \in (0,\infty]$,
    $$
    U_2(x)=0 \mbox{ for }x\geq 0,\quad U_2(x)\to +\infty\ (x\to -L_2) \ \mbox{\ \ and\ \ }    [\Lambda(U_2)]'(0-0) > -c''.
    $$
\end{enumerate}
\end{prop}

\begin{proof}
(i). Since $f(u)$ is a bistable nonlinearity in $u\in [s_1,1]$, the existence of the traveling wave $Q_{c_z}(x-c_z t)$ follows from the standard phase plane analysis as in reaction (linear) diffusion equation (cf. \cite{A2, AW1}).

(ii) and (iii). Using the properties in Lemma \ref{lem:3-TW}, one can constructs the small and big sharp waves as in \cite{A2}.

(iv). For $i=1$ or $i=3$, denote the trajectory starting from $R_i$ by $\Gamma^c_i$, and assume it contacts the negative $\psi$-axis at $(0,\psi^c_i)$, where $\psi^c_i\in [-\infty, 0)$. When $c'<c_s$, we know from Lemma \ref{lem:3-TW} that $-c' > \psi^{c'}_1 > \psi^{c'}_3$.
Therefore, for any $\varphi^*\in (\hat{\varphi}_1, 1)$, the trajectory passing $(\varphi^*,0)$ connects $(0,\psi^*)$ for some $\psi^* \in (\psi^{c'}_3, \psi^{c'}_1)$. It gives the profile of the desired traveling wave solution $U_1$.

(v). When $c''>c_s =c_z$, we know from Lemma \ref{lem:3-TW} that $-c'' < \psi^{c''}_1$.
Therefore, for any $\psi^*\in (-c'', \psi^{c''}_1)$, the trajectory starting from $(0,\psi^*)$ extends rightward and goes across the line $\{(\varphi,\psi) \mid \varphi =\hat{\varphi}_3, \psi<0\}$.
So it gives the profile of the desired traveling wave $U_2$.

This proves the conclusions.
\end{proof}

In addition, when $c=c_s$ and when we use the variable $\xi$, we find, as discussed in \cite{A2} for 
porous media equation, there are 
four types of trajectories on the $\varphi \psi$-phase plane.  The first type connects $R_1$ to a point on the negative $\psi$-axis, 
which gives the small sharp wave $Q_{c_s}$ (we call it a Type I traveling wave). The second 
and the third types of trajectories come from the positive $\psi$-axis, goes across the positive 
$\varphi$-axis and then tending to $R_0$ from lower right, one along the vector $(0,-1)$ and the 
other along $(-c_s, a_0 A_*)$ (recall that $R_0$ is a node). They give Type II traveling wave $\Phi_2(x-c_s t)$ 
and Type III $\Phi_3(x-c_s t)$ respectively:
$$
\left\{
\begin{array}{l}
\Phi'_2 (0) = +\infty, \quad \Phi'_2 (\xi)>0 \mbox{ in } \xi\in (0,\xi_2),\quad \Phi'_2(\xi)<0 \mbox{ in } x\in (\xi_2, \xi^*_2) 
\mbox{ and } \Phi_2(\xi^*_2 ) = \Phi'_2 (\xi^*_2) =0;\\
\Phi'_3 (0) = +\infty,\quad \Phi'_3(\xi)>0 \mbox{ in } \xi\in (0,\xi_3),\quad \Phi'_3(\xi)<0 \mbox{ and } \Phi_3(\xi)>0 \mbox{ for }\xi > \xi_3.
\end{array}
\right.
$$
The fourth type of trajectory comes from the lower right of $R_1$, goes leftward across
the domain $D$, and then goes down to infinity along the negative $\psi$-axis. It gives Type IV traveling wave $\Phi_4(x-c_s t)$ with 
$$
\Phi_4 (\xi) > \hat{\varphi}_1 \mbox{ for } \xi <0,\quad \Phi'_4 (\xi)<0 \mbox{ for } \xi \in (-\infty, \xi_4),
\quad \Phi_4 (\xi_4)=0 \mbox{ and } \Phi'_4 (\xi_4 -0) = -\infty.
$$


\end{document}